\documentclass[11pt, letterpaper]{amsart}

\usepackage{amssymb}
\usepackage{amsthm}
\usepackage{amsxtra}
\usepackage{graphicx}
\usepackage{color}
\usepackage{xcolor}
\usepackage{mathrsfs}

\newtheorem{theorem}{Theorem}

\newtheorem{lemma}[theorem]{Lemma}
\newtheorem{corollary}[theorem]{Corollary}

\theoremstyle{remark}
\newtheorem{remark}[theorem]{Remark}

\numberwithin{equation}{section}

\numberwithin{theorem}{section}

\numberwithin{table}{section}

\numberwithin{figure}{section}

\ifx\pdfoutput\undefined
  \DeclareGraphicsExtensions{.pstex, .eps}
\else
  \ifx\pdfoutput\relax
    \DeclareGraphicsExtensions{.pstex, .eps}
  \else
    \ifnum\pdfoutput>0
      \DeclareGraphicsExtensions{.pdf}
    \else
      \DeclareGraphicsExtensions{.pstex, .eps}
    \fi
  \fi
\fi

\title{Local Well-posedness for the Kinetic MMT Model}



\begin{document}

\author[P. Germain]{Pierre Germain}
\address{Pierre Germain, Department of Mathematics, Huxley Building, South Kensington Campus,
Imperial College London, London SW7 2AZ, United Kingdom}
\email{pgermain@ic.ac.uk}

\author[J. La]{Joonhyun La}
\address{Joonhyun La, June E Huh Center for Mathematical Challenges, Korea Institute for Advanced Study, 85 Hoegiro Dongdaemun-gu, Seoul 02455, Republic of Korea }
\email{joonhyun@kias.re.kr}

\author[Z. Zhang]{Katherine Zhiyuan Zhang}
\address{Katherine Zhiyuan Zhang, Department of Mathematics, Northeastern University, 43 Leon St., Boston, MA 02115, United States of America}
\email{zhi.zhang@northeastern.edu}

\maketitle

\begin{abstract}
The MMT equation was proposed by Majda, McLaughlin and Tabak as a model to study wave turbulence. We focus on the kinetic equation associated to this Hamiltonian system, which is believed to give a way to predict turbulent spectra. We clarify the formulation of the problem, and we develop the local well-posedness theory for this equation. Our analysis uncovers a surprising nonlinear smoothing phenomenon.
\end{abstract}

\tableofcontents

\section{Introduction}

The MMT equation, which was proposed as a toy model in~\cite{MMT}, has become very popular in the Physics and Applied Mathematics communities as a numerical testing ground for the predictions of wave turbulence theory. This model has a number of advantages: it is analytically tractable, easy to simulate, and yet it displays a very rich range of phenomena.

The aim of this paper is to establish the foundations of the mathematical analysis of the associated kinetic wave equation, in particular the existence and uniqueness of solutions.

\subsection{The MMT model and its associated kinetic equation}

We consider the MMT model proposed by Majda, McLaughlin and Tabak~\cite{MMT}:
\begin{equation} \label{MMT}
i{\partial_t} \psi = |D|^\alpha \psi + \lambda  |D|^{\beta/4} \big( \big| |D|^{\beta/4} \psi \big|^2 |D|^{\beta/4} \psi  \big) , \ \lambda = \pm 1,
\end{equation}
where $\alpha \in (0,1), \beta \in \mathbb{R}$ and $\psi = \psi(x, t)$ is a complex-valued wave field (a word of caution for the reader: our convention for the sign of $\beta$ is the opposite of~\cite{MMT}, but agrees with~\cite{ZGPD}).

The corresponding kinetic wave equation is given (see~\cite{MMT,ZGPD}) by
\begin{equation}
\label{KWE}
\tag{KWE}
{\partial_t} n(t,k) = \mathcal{C}(n(t))(k),
\end{equation}
with the collision operator given by
\begin{equation}
\begin{split}
&  \mathcal{C}(n)(k) = \int_{\mathbb{R}^{3 }} |k_1 k_2 k_3 k|^{\beta/2} ( n_1 n_2 n_3 + n_1 n_2 n_k - n_1 n_3 n_k - n_2 n_3 n_k ) \\
& \qquad\qquad\qquad\qquad\qquad \qquad \delta (k_1+ k_2-k_3-k) \delta (\omega_1 +\omega_2-\omega_3-\omega)\,  dk_1\, dk_2 \,dk_3,
\end{split}
\end{equation}
where we denoted
\begin{align*}
& n_k = n(t,k), \quad n_i=n(t,k_i), \quad i=1,2,3, \\
& \omega = |k|^\alpha, \quad \omega_i = |k_i|^\alpha, \quad i=1,2,3.
\end{align*}

The dependence of this model on two parameters, $\alpha$ and $\beta$, makes it possible to explore different regimes of wave turbulence. Like all authors in this subject, we will restrict most of the discussion to $\alpha = \frac{1}{2}$, which simplifies greatly the algebra.

Having set $\alpha = \frac{1}{2}$, what are significant values of $\beta$? The case $\beta = 0$ is deemed by Zakharov and collaborators~\cite{ZDP, ZGPD} to be particularly interesting for two reasons: first, there is no "nonlinear frequency shift" occuring in the Hamiltonian system; second, the signs of the energy and mass fluxes for the Kolmogorov spectrum suggest (see the discussion there) that the wave turbulence predictions should be most easily realized for this model. Furthermore, several authors~\cite{RB,CDO} have focused on this case to run numerical experiments.

\subsection{Physics of wave turbulence in the MMT model}

Classical introductions to wave turbulence can be found in the textbooks~\cite{ZLF, Nazarenko,Galtier}.

The MMT model was introduced by Majda, McLaughlin and Tabak in~\cite{MMT}, with the intention of verifying the predictions of wave turbulence theory on an easily computable 1D model. In some regimes however, their observations seemed to contradicted the classical theory, with the appearance of a new, so-called "MMT spectrum", see also~\cite{CMMT1,CMMT2}.

This discrepancy with the standard theory was examined analytically and numerically by Zakharov and collaborators, who concluded to the influence of nonlinear rigid structures, such as solitary waves, on the spectrum~\cite{ZGPD,ZDP,PZ}. Other authors~\cite{RB,CDO} emphasized the presence of features associated to fully developed turbulence for sufficiently strong forcing.

Very recently, B\"uhler and Du~\cite{BuhlerDu} performed new numerical experiments, and reported that the exponent predicted by the theory of wave turbulence is observed for the nonlinear spectrum, provided that the inertial range is sufficiently broad.

\subsection{Mathematical approaches to kinetic wave turbulence} 
There was recently major progress on the derivation of the kinetic wave equation from Hamiltonian dynamics~\cite{BuckmasterGermainHaniShatah,CollotGermain1,CollotGermain2,DengHani1} culminating in the first rigorous proof on a non-trivial kinetic time scale~\cite{DengHani2,DengHani3,DengHani4,StaffilaniTran,HannaniRosenzweigStaffilaniTran}. 

These remarkable developments establish for the first time the validity of the kinetic wave equation, which had been a matter of debate amongst physicists. As a consequence, the kinetic wave equation does describe turbulent behavior (in the appropriate regime), and further progress on turbulent dynamics could follow from a deeper understanding of the kinetic wave equation.

As far as rigorous Mathematics are concerned, the theory of the kinetic wave equation (in its many variants) is still in its infancy. Local well-posedness could be proved in some cases~\cite{NguyenTran,GIT}, and global weak solutions have been constructed in other cases~\cite{EV,SofferTran}. Of particular interest are the Kolmogorov-Zakharov solutions, whose stability was considered in~\cite{EscobedoMischlerVelazquez1,EscobedoMischlerVelazquez2,CDG}.

In the present article, we focus on the kinetic MMT equation, with several motivations: first, this is a standard model in the Physics and Applied Mathematics community. Second, it is arguably the simplest model for wave turbulence: the geometry of resonances and of the resonant manifold is very simple since it is one-dimensional, and the homogeneity of the equation allows to use scaling. Finally, it is the first kinetic wave equation to be studied in dimension one, opening the door to further developments in this setting.

\subsection{Main results and plan of the article} Section~\ref{sectionfirst} is dedicated to the review of some basic properties of the equation, such as scaling and monotone quantities; all computations are formal in this section. 

\medskip

In Section~\ref{sectioncollision}, we turn to giving a rigorous meaning to the collision operator $\mathcal{C}$, and obtaining a convenient parameterization. As is customary in the literature, this is done under the simplifying assumptions that $\alpha=\frac{1}{2}$, and that the function $n$ is even in $k$, which allows one to think of it as a function of the dispersion $\omega$: $n = n(\omega)$. These assumptions are made henceforth. 
 The parameterization obtained by Zakharov-Guyenne-Pushkarev-Dias~\cite{ZGPD} is first derived rigorously in~\eqref{ZGPDparameterization}; it is then observed that this expression can be symmetrized to obtain~\eqref{collision_rep}.

\medskip

In Section~\ref{sectionweightedLp}, we investigate the boundedness of the collision operator on the most simple and natural class of functional spaces, namely weighted $L^p$ spaces given by the norm
$$
\| f(\omega) \|_{L^p_{\theta,\gamma}(0,\infty)} = \| m(\omega) f(\omega) \|{L^p(0,\infty)}, \quad \mbox{where $m(\omega) = |\omega|^{-\theta} \mathbf{1}_{\omega < 1} + |\omega|^{\gamma} \mathbf{1}_{\omega > 1}$}.
$$
This leads to the following theorem (which follows from Lemma~\ref{lemma1} and Lemma~\ref{lemma2}); as far as positive results go, it suffers from the limitation to the range $\beta \in [-\frac{1}{2},0]$.

 \begin{theorem} 
 \begin{itemize}
 \item[(i)] The operator $\mathcal{C}$ is bounded on $L^\infty_{\theta,\gamma}$ if $-1/2 \leq \beta \leq 0$, $\gamma > 2\beta + 2$ and $\theta > -2\beta -1$
 \item[(ii)] The operator $\mathcal{C}$ is not bounded on $L^p_{\theta,\gamma}$ if $p<3$ and $\gamma,\theta \in \mathbb{R}$.
 \end{itemize} 
 \end{theorem}

\begin{corollary}
The equation~\eqref{KWE} is locally well posed in $L^\infty_{\theta,\gamma}$ if $-1/2 \leq \beta \leq 0$, $\gamma > 2\beta + 2$ and $\theta > -2\beta -1$.

More specifically, for data in $L^\infty_{\theta,\gamma}$, there exists $T>0$ and a unique solution in $\mathcal{C}([0,T],L^\infty_{\theta,\gamma})$ which depends continuously on the data. The equation $\partial_t n = \mathcal{C}(n)$ is understood in the sense of equality in $\mathcal{C}([0,T],L^\infty_{\theta,\gamma})$. \end{corollary}

To go further, we consider in Section~\ref{sectionscaling} smoother data, namely such that
$$
N(\omega) = \omega^{2\beta + \frac{3}{2}} n(\omega)
$$ 
belongs to 
$$
X = \{ N \in L^\infty ((0, \infty )), \, DN (\omega) := \omega N'(\omega) \in L^{2p_0} ((0, \infty  )) \cap L^\infty((0, \infty )) \}
$$
for $\beta$ and $p_0$ satisfying \eqref{betaanddata}. Note that the exponent $2\beta + \frac{3}{2}$ appearing in the definition of $N$ corresponds to the scaling of the equation, so that the functional framework is close to being scaling invariant, and as such, optimal.

We first state the a priori estimate.
\begin{theorem} \label{apriori}
Suppose that $p_0, \beta$, and $N$ satisfy the technical assumptions (A1)-(A4) (see Section \ref{sectionscaling}). Then there exists a time $T = T(p_0, \beta, N_0)$ and a constant $\mathbf{C} (p_0, \beta, N_0)$ such that
\begin{equation}
\| N \|_{L^\infty (0, T; X) } \le \mathbf{C}(p_0, \beta, N_0).
\end{equation}
Also, we have the gain of regularity \eqref{Gainsimplified1} and in particular, for $\beta \in (-3/4, 3/4)$, 
\begin{equation} \label{gainofregularityestimate}
\int_0 ^T [ DN (t) ]_\beta ^2 dt \le C(N_0, \beta, p_0 ) < \infty, 
\end{equation}
where
\begin{equation} \label{fractionalSobolevnorm}
\begin{split}
[ F ]_\beta ^2   &:= \int_0 ^\infty \int_0 ^\infty  \mathbf{1}_{[2/3, 3/2]} \left ( \frac{\omega'}{\omega } \right ) \frac{| F(\omega) - F(\omega') |^2}{|\omega' - \omega|^{\beta+1} } (\omega \omega')^{\beta/2} d\omega d\omega' \\
&+ \int_0 ^\infty \int_0 ^\infty  \mathbf{1}_{[1/2, 2]} \left ( \frac{\omega'}{\omega } \right ) \frac{| F(\omega) - F(\omega') |^2}{|\omega' - \omega|^{-2 \beta } } (\omega \omega')^{- (2\beta+1)/2 } d\omega d\omega' .
\end{split}
\end{equation}
\end{theorem}
Note that the seminorm \eqref{fractionalSobolevnorm} is reminiscent of the fractional Sobolev norm in its integral formulation\footnote{ Recall that $$
\| u \|_{H^s(\mathbb{R})}^2 = \| u \|_{L^2}^2 + \iint \frac{|u(x)-u(y)|^2}{|x-y|^{1+2s}}\,dx\,dy \qquad \mbox{if $s < \frac{1}{2}$},
$$ see for example, \cite{DINEZZA2012521}.}.

Next, we state the following local well-posedness theorem.
\begin{theorem} \label{construction}
For $p_0, \beta$ satisfying \eqref{betaanddata}, the Cauchy problem for \eqref{Neqn} with initial data $N_0 \in X$, $ \inf_{\omega \in (0, \infty) } N_0 (\omega) > 0$ admits the unique strong solution $N \in L^\infty (0, T; X)$ for a time $T = T(N_0, p_0, \beta) > 0$. Moreover, $\inf_{\omega \in (0, \infty ), t \in (0, T) } N(t, \omega) > 0$

{\color{black} The solution is either understood in the mild sense, or in the sense of pointwise equality $\partial_t N = \overline{\mathcal{C}}(N)$, which holds almost everywhere in $t$ for each $\omega >0$. Indeed, the field $N$ is continuous and almost everywhere differentiable.}
\end{theorem}

Theorem \ref{apriori} reveals an unexpected nonlinear smoothing mechanism which is tied to the structure of the collision operator. Such a phenomenon is already known to occur in some Boltzmann type equations (see the discussion below), but it came as a surprise in the context of 1D kinetic wave equation. We believe it might open interesting research directions.

On the other hand, we remark the gap between Theorem \ref{construction} and Theorem \ref{apriori}: Theorem \ref{construction} does not conclude the gain of regularity \eqref{Gainsimplified1} and \eqref{Gainsimplified2}. As one can see from the proof, the main reason is that we do not know how to obtain further regularity from these gains: more precisely, we only have weak convergence of derivatives of approximate solutions, not pointwise. To extract pointwise convergence, we need to employ regularity properties of the derivative, but it is not clear at this point how to do this. Development of regularity theory for the kinetic MMT equation would enable us to recover \eqref{Gainsimplified1} and \eqref{Gainsimplified2}.

\subsection{Perspectives}
\noindent \underline{Analogy with non-cutoff Boltzmann situation.} There is a very interesting analogy with the regularizing effect of Boltzmann equation with non-cutoff kernel (for example, \cite{Alexandre2010}, \cite{10.1215/21562261-1625154}, \cite{Gressman_2011} and also \cite{Jang2022} for relativistic Boltzmann.) In addition to the global well-posedness in smooth, small data regime, the gain of regularity has been shown in \cite{Gressman_2011}. The gain of regularity was measured in the following anisotropic norm:
\begin{equation} \label{noncutoffboltzmannnorm}
|f|_{N^{s, \gamma} }^2 := \int_{\mathbb{R}^n}\langle v \rangle ^{\gamma + 2s} |f(v) |^2 dv + \int_{\mathbb{R}^n} \int_{\mathbb{R}^n} ( \langle v \rangle \langle v' \rangle )^{\frac{\gamma + 2s + 1}{2}} \frac{| f(v) - f(v') |^2}{d(v, v')^{n+2s} } \mathbf{1}_{d(v, v') \le 1}  dv dv',
\end{equation}
where $\langle v \rangle = \sqrt{1+|v|^2}$ and $$d(v, v') = \sqrt{|v-v'|^2 + \frac{1}{4}(|v|^2 - |v'|^2 )^2 }$$
(see \cite{Gressman_2011} for further details.) They proved that if initial data is in a certain weighted Sobolev space (in $x, v$ variable), then not only solution propagates in that space, its derivatives gain the regularity measured by \eqref{noncutoffboltzmannnorm}. The norm in \eqref{noncutoffboltzmannnorm} is quite analogous to the norm \eqref{fractionalSobolevnorm} presented in the paper. Besides, in a series of works including \cite{Alexandre2010} and \cite{10.1215/21562261-1625154}, the authors showed that the solution in fact acquires $\mathcal{C}^\infty$-regularity. Therefore, it would be naturally interesting to ask if similar $\mathcal{C}^\infty$-smoothing effect could be shown for the kinetic wave equation as well.

\medskip

\noindent \underline{Smoothness of solutions.} Based on the regularizing effect presented above, it seems natural to conjecture that the solution to~\eqref{KWE} becomes positive and smooth for positive time - as long as it exists.

\medskip

\noindent \underline{Global solutions.} Obtaining global solutions for~\eqref{KWE} seems like a very challenging question. Indeed, the kernel of the collision operator is very singular, and this seems to preclude any possibility to make sense of it for locally $L^1$ fields, let alone prove local well-posedness at such a low regularity.

\subsection{Notations} For two quantities $A,B$, we write $A \lesssim B$ if there exists a constant $C>0$ such that $A \leq CB$. If $A \lesssim B$ and $B \lesssim A$, we denote $A \sim B$

\subsection*{Acknowledgements.} When preparing this article, Pierre Germain was supported by the Simons Foundation Collaboration on Wave Turbulence, a start up grant from Imperial College, and a Wolfson fellowship. Joonhyun La was supported by June Huh fellowship from KIAS. Katherine Zhiyuan Zhang was supported by the Simons Foundation Collaboration on Wave Turbulence and by an AMS-Simons travel grant. The authors thank Jin Woo Jang for fruitful discussions.

\section{First properties of the equation}

\label{sectionfirst}

In this section, we review some of the formal properties of the equation: scaling, monotone quantities, explicit solutions, and deduce the expected behavior of the equation for various ranges of parameters.

\subsection{Monotone quantities}

For any function $F(k)$, we get after symmetrizing that
\begin{align*}
\frac{d}{dt} \int F(k) n_k\,dk & = - \frac{1}{4} \int_{\mathbb{R}^{4}} |k_1 k_2 k_3 k|^{\beta/2} ( n_1 n_2 n_3 + n_1 n_2 n_k - n_1 n_3 n_k - n_2 n_3 n_k ) \\
& \qquad \delta (k_1+ k_2-k_3-k) \delta (\omega_1 +\omega_2-\omega_3-\omega)(F_1 + F_2 - F_3 - F)\,  dk_1\, dk_2 \,dk_3\,dk.
\end{align*}
Choosing $F(k) = 1,k,|k|^\alpha$, this gives the conservation of the mass, momentum, and energy
$$
\int n_k\,dk, \qquad \int k n_k\,dk, \qquad \int \omega_k n_k\,dk.
$$
Similarly, we get after symmetrizing that
\begin{align*}
 &\frac{d}{dt} \int F(n_k) \,dk = - \frac{1}{4} \int_{\mathbb{R}^{4}} |k_1 k_2 k_3 k|^{\beta/2} ( n_1 n_2 n_3 + n_1 n_2 n_k - n_1 n_3 n_k - n_2 n_3 n_k ) \\
& \qquad \delta (k_1+ k_2-k_3-k) \delta (\omega_1 +\omega_2-\omega_3-\omega)(F'(n_1) + F'(n_2) - F'(n_3) - F'(n_k))\,  dk_1\, dk_2 \,dk_3\,dk.
\end{align*}
Choosing $F(n) = \log n$, this gives the H-theorem
\begin{align*}
&\frac{d}{dt} \int \log n_k \,dk \geq 0.
\end{align*}

\subsection{Changing the independent variable from $k$ to $\omega$}

It is customary and convenient to adopt $\omega = \sqrt{k}$ as the independent variable. For notational simplicity in this new set of coordinates, we will assume from now on that \underline{$n(k)$ is even}, so that we can choose $n(\omega)$, $\omega>0$, as the unknown function. 

We will still denote $n$ and $\mathcal{C}$ for the unknown function, and the collision operator; it will be obvious from the context whether these are considered as functions of $k$ or $\omega$.

Up to time rescaling by a constant, the equation is now given by
$$
{\partial_t} n(t,\omega) = \mathcal{C}(n(t))(\omega),
$$
with the collision operator
\begin{equation}
\begin{split}
\mathcal{C}(n)(\omega) &  = \sum_{\epsilon_1,\epsilon_2,\epsilon_3 \in \{\pm 1\}}\int_{\mathbb{R}_+^{3 }} \omega^{\frac{\beta}{2\alpha}} (\omega_1 \omega_2 \omega_3)^{\frac{\beta}{2\alpha}-1+ \frac{1}{\alpha}} ( n_1 n_2 n_3 + n_1 n_2 n_\omega - n_1 n_3 n_\omega - n_2 n_3 n_\omega ) \\
& \qquad\qquad\qquad \delta (\epsilon_1 \omega_1^{1/\alpha} +\epsilon_2 \omega_2^{1/\alpha} + \epsilon_3 \omega_3^{1/\alpha} - \omega^{1/\alpha} ) \delta (\omega_1 +\omega_2-\omega_3-\omega)\,  d\omega_1\, d\omega_2 \,d\omega_3,
\end{split}
\end{equation}
where we denoted
\begin{align*}
& n_\omega = n(t,\omega), \quad n_i=n(t,\omega_i), \quad i=1,2,3.
\end{align*}
In these new coordinates, the conserved quantities become
$$
\int_{\mathbb{R}_+} n(\omega) \omega^{\frac{1}{\alpha}-1} \,d\omega, \qquad \int_{\mathbb{R}_+} n(\omega) \omega^{\frac{1}{\alpha}} \,d\omega
$$
(the momentum being identically zero by parity).

\subsection{Scaling considerations} 

\subsubsection{Scaling around the origin} 
A computation shows that the collision operator satisfies the scaling relation
$$
\mathcal{C}(n(\lambda \cdot))(\omega) = \lambda^{-2\frac{\beta}{\alpha} + 1 - \frac{2}{\alpha}} \mathcal{C} (n) (\lambda \omega).
$$
Thus, the following transformation leaves the set of solutions invariant
$$
n(t,\omega) \mapsto \lambda^{\frac{\beta}{\alpha} + \frac{1}{\alpha} - \frac{1}{2}} n(t,\lambda \omega)
$$
(there are also scaling transformations involving the time variable, but they are less relevant, and we ignore them here).

\subsubsection{Formation of a condensate} This scaling law suggests that the formation of a condensate at frequency 0 is only possible if $2\beta < \alpha$. 

The heuristic argument is as follows: suppose that, along a sequence of times $(t_n)$ converging to 1, $n(\omega)$ converges (locally) to a Dirac $\delta$ at the origin. If this happens in a self-similar way, there exists a profile $\psi$ such that $n(t_n,\omega) = \ell_n^{1/\alpha} \psi(\ell_n \omega)$, where $\ell_n \to \infty$. If the profile $\psi$ is, say, smooth, there exists a smooth solution $N$ on $[0,T]$, with $T>0$, of~\eqref{KWE} with data $\psi$. We now use the scaling law above to get a solution on $[t_n,t_n+T]$ for~\eqref{KWE} with data $\ell_n^{\frac{\beta}{\alpha} + \frac{1}{\alpha} - \frac{1}{2}} \psi(\ell_n \omega)$. 
As long as $\frac{1}{\alpha} < \frac{\beta}{\alpha} + \frac{1}{\alpha} - \frac{1}{2}$, which is equivalent to $2\beta > \alpha$, the smaller data $\ell_n^{1/\alpha} \psi(\ell_n \omega)$ should also have a smooth solution on $[t_n,t_n+T]$; but this contradicts the formation of a singularity at time $1$.

\subsection{Stationary solutions}
A family of stationary solutions is provided by the \textit{Rayleigh-Jeans} spectra
$$
\frac{1}{c_1 |k|^\alpha + c_2} = \frac{1}{c_1 \omega + c_2}, \qquad c_1,c_2 \geq 0.
$$

While these solutions correspond to equilibria, the \textit{Kolmogorov-Zakharov spectra} give (direct or inverse) cascades of mass or energy. They are the following power laws~\cite{MMT,ZGPD}
$$
\omega^{-\frac{2\beta + 3 - \alpha}{3\alpha}}, \qquad \omega^{-\frac{2\beta + 3}{3 \alpha}}.
$$

\section{The collision operator for $\alpha = \frac{1}{2}$}

\label{sectioncollision}

\textit{From now on, we focus on the case $\alpha=1/2$.} The main advantage of this particular case is that it becomes possible to find an explicit parameterization of the collision operator, which will simplify its treatment. The aim of this section is to derive this explicit parameterization; we will also explain carefully how to give a precise meaning to the expression giving the collision operator.

Recall that the collision operator is given by
\begin{equation}
\label{formulaC}
\begin{split}
\mathcal{C}(n)(\omega) &  = \sum_{\epsilon_1,\epsilon_2,\epsilon_3 \in \{\pm 1\}}\int_{\mathbb{R}_+^{3 }} \omega^{\beta} (\omega_1 \omega_2 \omega_3)^{\beta+1} ( n_1 n_2 n_3 + n_1 n_2 n_\omega - n_1 n_3 n_\omega - n_2 n_3 n_\omega ) \\
& \qquad\qquad\qquad \delta (\epsilon_1 \omega_1^{2} +\epsilon_2 \omega_2^{2} + \epsilon_3 \omega_3^{2} - \omega^{2} ) \delta (\omega_1 +\omega_2-\omega_3-\omega)\,  d\omega_1\, d\omega_2 \,d\omega_3,
\end{split}
\end{equation}

\subsection{Around the coarea formula}
The integral giving $\mathcal{C}$ features a product of $\delta$ functions; to show that it is well-defined, and interpret it appropriately, the coarea formula turns out to be the key tool. The discussion will be simplified by introducing some shorthands: denoting 
\begin{align*}
& F_1(\omega_1,\omega_2,\omega_3) = \epsilon_1 \omega_1^{2} +\epsilon_2 \omega_2^{2} + \epsilon_3 \omega_3^{2} - \omega^{2} \\
& F_2(\omega_1,\omega_2,\omega_3) = \omega_1 +\omega_2-\omega_3-\omega \\
&  \mathcal{I}(\omega_1,\omega_2,\omega_3) =  n_1 n_2 n_3 + n_1 n_2 n_\omega - n_1 n_3 n_\omega - n_2 n_3 n_\omega,
\end{align*}
we are dealing with a sum of terms of the type
\begin{equation}
\label{aterm}
\int \mathcal{I}(\omega_1,\omega_2,\omega_3) \delta(F_1(\omega_1,\omega_2,\omega_3)) \delta(F_2(\omega_1,\omega_2,\omega_3)) \,d\omega_1  \,d\omega_2  \,d\omega_3
\end{equation}
(thinking of $\omega$ as being fixed; the dependence of $F_2$ on $\epsilon_1,\epsilon_2,\epsilon_3$ is implicit).
Assuming that $n \in \mathcal{C}_0^\infty(0,\infty)$ for simplicity (even though a sufficiently fast decay at $0$ and $\infty$ for $\varphi$ would of course be sufficient), the integrand $\mathcal{I}$ is smooth and compactly supported. By the coarea formula\footnote{see the Encyclopedia of Mathematics, 
http:\/\/encyclopediaofmath.org\/index.php?title=Coarea\_formula{\&}oldid=29308
}, this allows to write the above as
\begin{equation}
\label{formulajacobian}
\eqref{aterm} = \int_{F_1^{-1}(\{0\}) \cap F_2^{-1}(\{0\})} \mathcal{I}(\omega_1,\omega_2,\omega_3) \frac{1}{|J(F_1,F_2)|} d \mathcal{H}^1,
\end{equation}
where the Jacobian of $(F_1,F_2)$ is given by
$$
J(F_1,F_2) = \left[ \det \begin{pmatrix} |\nabla F_1|^2 & \nabla F_1 \cdot \nabla F_2 \\  \nabla F_1 \cdot \nabla F_2 & |\nabla F_2|^2 \end{pmatrix} \right]^{\frac 12}
$$
the Jacobian of $(F_1,F_2)$ (which does not vanish), and $d \mathcal{H}^1$ is the measure induced on $F_1^{-1}(\{0\}) \cap F_2^{-1}(\{0\})$ by the Lebesgue measure in $\mathbb{R}^3$. This is well-defined as long as $\nabla F_1$ and $\nabla F_2$ do not align, or in other words as long as the manifolds $\{ F_1 = 0 \}$ and $\{ F_2 = 0 \}$ intersect transversally.

\medskip

The discussion above guarantees that the collision operator is well-defined; but to compute its parametrization, it is more convenient to first use $\delta(F_2)$ to replace $\omega_3$ by $\omega_1 + \omega_2 - \omega$:
$$
\eqref{aterm} = \int \widetilde{\mathcal{I}}(\omega_1,\omega_2) \delta(\widetilde{F_1}(\omega_1,\omega_2)) \,d\omega_1  \,d\omega_2.
$$
Here, the integration variables are $\omega_1,\omega_2$, and $\widetilde{\mathcal{I}}(\omega_1,\omega_2) = \mathcal{I}(\omega_1,\omega_2,\omega_1+\omega_2-\omega)$, similarly for $\widetilde{F_1}$. By the coarea formula, this becomes
\begin{equation}
\label{coarea2}
\eqref{aterm} = \int_{\widetilde{F_1}^{-1}(\{ 0 \})} \widetilde{\mathcal{I}}(\omega_1,\omega_2) \frac{1}{|\nabla \widetilde{F_1}|} d \mathcal{H}^1.
\end{equation}

Suppose now that $\varphi(u) = (\omega_1 (u), \omega_2(u))$, $u \in [a,b]$ is a parametrization of $\widetilde{F_1}^{-1} (0)$. Then we claim that
\begin{equation}
\label{simplecoarea}
\int_{\widetilde{F_1}^{-1} (0) } \widetilde{\mathcal{I}} (\omega_1, \omega_2) \frac{d \mathcal{H}^1} {| \nabla \widetilde{F_1} |} = \int_a ^b \widetilde{\mathcal{I}}(u) \frac{|\omega_1' (u) |}{|\partial_{\omega_2} \widetilde{F_1} |} du = \int_a ^b F(u) \frac{|\omega_2' (u) |}{|\partial_{\omega_1} \widetilde{F_1} |} du.
\end{equation}

We first note that $\frac{d \mathcal{H}^1} {| \nabla \widetilde{F_1} |} = \frac{ |\varphi'(u)|}{{| \nabla \widetilde{F_1} |} } du$. However, $\frac{d}{du} \widetilde{F_1} ( \omega_1(u), \omega_2 (u) ) = \varphi' \cdot \nabla \widetilde{F_1} = 0$ since $\varphi $ is embedded in the level set of $\widetilde{F_1}$. Therefore, $\nabla \widetilde{F_1}$ is parallel to $(\varphi') ^{\perp} = (-\omega_2 ', \omega_1 ')$, and in particular, 
$$
\frac{|\varphi'(u) |}{| \nabla \widetilde{F_1}  |} = \frac{ |\varphi'(u)|^2}{ | (\varphi')^{\perp} \cdot \nabla \widetilde{F_1} | }.
$$
However, from $ \varphi' \cdot \nabla \widetilde{F_1} = 0$ we have $(\varphi')^{\perp} \cdot \nabla \widetilde{F_1}=\frac{ |\varphi'(u)|^2 }{\omega_1'(u)} \partial_{\omega_2} \widetilde{F_1}$, and the result follows.

\subsection{Trivial resonances} 

First, we claim that the above expression vanishes in either of the following cases
\begin{align*}
(\epsilon_1 , \epsilon_2, \epsilon_3) = (-,-,-), \; (+,-,+), \; (-,+,+), \; (+,+,-). 
\end{align*}
This follows from examining the resonant manifold
\begin{equation}
\label{RM}
\epsilon_1 \omega_1^{2} +\epsilon_2 \omega_2^{2} + \epsilon_3 \omega_3^{2} - \omega^{2} = \omega_1 +\omega_2-\omega_3-\omega = 0, \qquad \omega_1,\omega_2,\omega_3,\omega \geq 0.
\end{equation}
Indeed:
\begin{enumerate}
\item If $(\epsilon_1 , \epsilon_2, \epsilon_3) = (-,-,-)$, then $\omega_1^2 + \omega_2^2 + \omega_3^2 + \omega^2 = 0$ only has a trivial solution.
\item If $(\epsilon_1, \epsilon_2, \epsilon_3) = (+,-,+)$, then~\eqref{RM} implies that $(\omega_1,\omega_2) = (\omega,\omega_3)$, in which case the integrand of~\eqref{formulaC} vanishes.
\item The case $(\epsilon_1, \epsilon_2, \epsilon_3) = (-,+,+)$ is symmetrical.
\item If $(\epsilon_1,\epsilon_2,\epsilon_3) = (+,+,-)$, then~\eqref{RM} implies that $\{ \omega_1,\omega_2 \} = \{ \omega,\omega_3 \}$, and once again the integrand of~\eqref{formulaC} vanishes.
\end{enumerate}

This discussion was mostly formal; we now resort to~\eqref{formulajacobian} to give it a more rigorous framework. In the case (2), it is easy to see that the gradients of $F_1(\omega_1,\omega_2,\omega_3)$ and $F_2(\omega_1,\omega_2,\omega_3) = \omega_1^2 - \omega_2^2 + \omega_3^2 - \omega^2$ are never colinear on the resonant manifold where $(\omega_1,\omega_2) = (\omega_3,\omega_4)$. The above conclusion is therefore verified.

However, in the case (4), the gradients of $F_1(\omega_1,\omega_2,\omega_3)$ and $F_2(\omega_1,\omega_2,\omega_3) = \omega_1^2 + \omega_2^2 - \omega_3^2 - \omega^2$ are colinear at the point $\omega_1 = \omega_2 = \omega_3 = 
\omega_4$, which belongs to the resonant manifold. Furthermore, the singularity caused by the Jacobian in the denominator in~\eqref{formulajacobian} is not integrable; even though the integrand is zero, this term appears ill-defined. To settle matters, we will use a regularization procedure, and show that it has zero limit.

More precisely, we will interpret $\delta(\omega_1 + \omega_2 - \omega_3 - \omega_4)$ as
$$
\delta(\omega_1 + \omega_2 - \omega_3 - \omega_4) = \lim_{\delta\rightarrow 0^+} \frac{1}{2\delta} \mathbf{1}_{(-\delta, \delta) } (\omega_1 + \omega_2 - \omega_3 - \omega_4).
$$
We will furthermore denote
$$
\omega_1 + \omega_2 - \omega_3 - \omega_4 = t
$$
and, under this constraint,
$$
\Phi_t = \omega_1^2 + \epsilon_2 \omega_2^2 - \epsilon_3 \omega_3^2 - \omega^2
$$
 is a function of two of $(\omega_1, \omega_2, \omega_3)$. By the coarea formula, we are thus led to investigate the limit as $\delta$ goes to zero of
\begin{equation*}
 \lim_{\delta \rightarrow 0^+} \frac{1}{2\delta}  \int_{-\delta} ^{\delta} \int_{\Phi_t ^{-1} (0) }(\omega_1 \omega_2 \omega_3)^{\beta+1} \left[n_1 n_2 n_3 + n_1 n_2 n_\omega - n_1 n_3 n_\omega - n_2 n_3 n_\omega \right] \frac{d \mathcal{H}^1} {| \nabla_{\omega_i} \Phi_{t} | } dt,
\end{equation*}
assuming for simplicity that $n \in \mathcal{C}_0^\infty$, where $d \mathcal{H}^1$ is the measure for the curve $ \{ \Phi_t = 0 \} \cap \{ \sum_{i=1} ^4 (-1)^i \omega_i = t \}$ and $\nabla_{\omega_i} \Phi_t$ is the derivative of $\Phi_t$ with respect to the two variables on which it depends. 

The following formula will prove useful: if $(\omega_1(u),\omega_2(u))$ is a parameterization of $\Phi_t^{-1}(\{0\})$, then
$$
\int_{\Phi_{t} ^{-1} (0) } F (\omega_1, \omega_2) \frac{d \mathcal{H}^1} {| \nabla_{\omega_1, \omega_2} \Phi_{t}  |} = \int_a ^b F(u) \frac{|\omega_1' (u) |}{|\partial_{\omega_2} \Phi_{t} |} du = \int_a ^b F(u) \frac{|\omega_2' (u) |}{|\partial_{\omega_1} \Phi_{t} |} du;
$$
it can be proved like~\eqref{simplecoarea}.
Observe that
\begin{equation}
|\nabla_{\omega_2} \Phi_t |  = 2 |\omega_2 - \omega_3|.
\end{equation}
We have the following:
\begin{equation}
\begin{split}
|\omega_2 - \omega_3| &= |\omega_1 - \omega - t|, \\
|\omega_2 ^2 - \omega_3 ^2| &= |\omega_1 ^2 - \omega^2|,
\end{split}
\end{equation}
which implies the following:
\begin{equation}
|\omega_1 - \omega - t | = \frac{\omega_1 + \omega}{\omega_2 + \omega_3} |\omega_1 - \omega|. 
\end{equation}
An important consequence is that since $n$ is compactly supported, $|\omega_1 - \omega - t|$ cannot be too small: if $|\omega_1 - \omega - t| < ct$ for a sufficiently small $c<1$, $|\omega_1 - \omega| > C' t$ for some $C'>1/2$, so $\omega_2 + \omega_3 = |\omega_1 - \omega| (\omega_1 + \omega) / |\omega_1 - \omega - t| > C''\omega$: so if $c<1$ is sufficiently small, both $\omega_2$ and $\omega_3$ will lie outside the support of $n$ (since $|\omega_2 - \omega_3| = |\omega_1 - \omega - t|$ small), and thus $\sum_{j=1} ^4 (-1)^j \frac{\prod_{i=1} ^4 n_i} {n_j} = 0$ there. Furthermore, $|\omega_1 - \omega|$ cannot be too large as well: we have 
\begin{equation}
(\omega_1 - \omega) =
\begin{cases}
\frac{(\omega_2 + t)^2 - \omega_2 ^2}{2 (\omega+ \omega_2 + t) }, \text{ case }\qquad (+, +, -), \\
\frac{\omega_2^2 - (\omega_2 + t)^2}{2 (\omega_1 - \omega_2 - t)}, \text{ case } \qquad (+, -, +),
\end{cases}
\end{equation}
from the following calculation:\begin{equation}
\begin{split}
0 & = \omega_1^2 + \omega_2^2 - \omega_3^2 - \omega^2\\
&= (\omega_1 - \omega) (\omega_1 + \omega) - (\omega_1 -\omega - t) (2 \omega_3 + \omega + t - \omega_1), \\
- (2 \omega_2 t + t^2) &= (\omega_1 - \omega )\big [ 2 \omega_1 - 2 (\omega_3 + t)  \big ]
\end{split}
\end{equation}
which gives the result.
If $|\omega_1 - \omega_2 - t| \ge |\omega_3 - \omega_2 - t|$, since $\omega_2 - \omega_3 - t = - \omega - \omega_1$, we have
\begin{equation}
|\omega_1- \omega| \le \frac{ C |t| }{|\omega_1 - \omega|},
\end{equation}
that is, $|\omega_1 - \omega| \le C \sqrt{|t|}$. In both cases, for a sufficiently small $t$, we have $|\omega_1 - \omega - t| \le C \sqrt{|t|}$. Therefore, if $(\omega_1, \omega_2, \omega_3) \in \Phi_t^{-1} (0)$, then only those $\omega_1$ such that
\begin{equation} \label{boundomegaomega1} c |t| < |\omega_1 - \omega - t| < C \sqrt{|t|} \end{equation}
yield nonzero integrands. On the other hand, we have
\begin{equation} \label{canceldom}
\left | n_1 n_2 n_3 + n_1 n_2 n_\omega - n_1 n_3 n_\omega - n_2 n_3 n_\omega \right | \le n_1 n_4 |n_2 - n_3| + n_2 n_3 |n_1 - n_4| \le C( |\omega_1 - \omega - t| + |t|).
\end{equation}
Finally, we remark that from the upper bound in \eqref{boundomegaomega1}, both $\omega$ and $\omega_1$ are in $[c^{-1}, c]$ for some $c>1$ for $(\omega_1, \omega_2, \omega_3) \in \Phi_t^{-1} (0)$ for small $t$. This implies that $|\omega_3 - \omega_2 | = |\omega_1 - \omega - t| \in [0, C]$ for some $C>0$, and therefore either $\omega_3, \omega_2 \in [0, C]$ for some $C>0$ or $\omega_3, \omega_2$ is greater than the upper bound for $\text{supp }n$.

Summarizing the above, we see that if we parametrize $\Phi_t^{-1} (0)$ with $\omega_1$, we have
\begin{equation}
\begin{split}
&\left |\int_{\Phi_t ^{-1} (0) }(\omega_1 \omega_2 \omega_3)^{\beta+1} \left[n_1 n_2 n_3 + n_1 n_2 n_\omega - n_1 n_3 n_\omega - n_2 n_3 n_\omega \right] \frac{d \mathcal{H}^1} {| \nabla_{\omega_i} \Phi_{t} | } \right | \\
& \qquad \lesssim \int_{c|t| < |\omega_1 - \omega - t| < C\sqrt{|t|} } \frac{|\omega_1 - \omega - t| + |t| }{|\omega_1 - \omega - t| } d\omega_1 < C \sqrt{|t|},
\end{split}
\end{equation}
which gives the desired result after integrating over $[-\delta,\delta]$ and letting $\delta \to 0$.

\subsection{Non-trivial resonances: the Zakharov-Guyenne-Pushkarev-Dias parameterization} 
In the previous subsection, we learned that trivial resonances can be eliminated from the sum defining $\mathcal{C}$, which becomes therefore
\begin{equation} \label{MMT-kwe}
\begin{split}
 \partial_t n_\omega & = \int_{\mathbb{R}_+^{3 }}\omega^{\beta} (\omega_1 \omega_2 \omega_3)^{\beta+ 1}  ( n_1 n_2 n_3 + n_1 n_2 n_\omega - n_1 n_3 n_\omega - n_2 n_3 n_\omega ) \\
& \qquad \cdot \delta (\omega_1 +\omega_2-\omega_3-\omega) 
[ \delta (\omega_1^{2} + \omega_2^{2} +\omega_3^{2}-\omega^{2}) + \delta (\omega_1^{2} + \omega_2^{2} -\omega_3^{2} + \omega^{2}) \\
& \qquad\qquad + \delta (\omega_1^{2} - \omega_2^{2} + \omega_3^{2} +\omega^{2}) + \delta (- \omega_1^{2} + \omega_2^{2} + \omega_3^{2} + \omega^{2}) ] \, d\omega_1 \, d\omega_2 \,d\omega_3  \\
& = \mathcal{C}_1 (n, n, n) + \mathcal{C}_2 (n, n, n) + \mathcal{C}_3 (n, n, n) + \mathcal{C}_4 (n, n, n) .  \\
\end{split}
\end{equation}
In the present subsection, we turn to obtaining a parameterization for the collision operator: we will show that
\begin{equation}
\label{ZGPDparameterization}
\mathcal{C} = \mathcal{S}_1 + \mathcal{S}_2 + \mathcal{S}_3 + \mathcal{S}_4
\end{equation}
where
\begin{align*}
\mathcal{S}_1(n)(\omega) & = \omega^{4\beta + 3}  \int_0^1 \frac{u^{2\beta +2} (u+1)^{2\beta+2}}{(1+u+u^2)^{3\beta+4}}  ( f_1 f_2 f_3 + f_1 f_2 f_\omega - f_1 f_\omega f_3 - f_\omega f_2 f_3  )\, du \\
\mathcal{S}_2(n)(\omega) & = \omega^{4\beta + 3}  \int_0^1  \frac{(1+u+u^2)^{\beta+1} (u+1)^{2\beta+2}}{u^{2\beta +3} }  ( f_1 f_2 f_3 + f_1 f_2 f_\omega - f_1 f_\omega f_3 - f_\omega f_2 f_3  )\, du \\
\mathcal{S}_{3} (f)(\omega) &  =    \omega^{4\beta+3}  \int_0^1 \frac{(1+u+u^2)^{\beta+1}u^{2\beta+2}}{(u+1)^{2\beta+3}}  ( f_1 f_2 f_3 + f_1 f_2 f_\omega - f_1 f_\omega f_3 - f_\omega f_2 f_3  ) \, du \\
\mathcal{S}_{4} (f)(\omega) & =  \omega^{4\beta+3}  \int_0^1 \frac{(1+u+u^2)^{\beta+1}}{u^{2\beta+3} (u+1)^{2\beta+3}}  ( f_1 f_2 f_3 + f_1 f_2 f_\omega - f_1 f_\omega f_3 - f_\omega f_2 f_3  ) \, du.  
\end{align*}

For each of the above integrals, the notation $f_i$ stands for $f(u_i \omega)$, where the $u_i$ are given as follows
\begin{align*} 
& \mbox{for $\mathcal{S}_1$}:  \quad u_1 = \frac{1+u}{1+u+u^2},  \quad  u_2 = \frac{u(1+u)}{1+u+u^2}, \quad  u_3 = \frac{u}{1+u+u^2} \\
& \mbox{for $\mathcal{S}_2$}: \quad u_1 = u+1 , \quad   u_2 = \frac{u+1}{u} , \quad  u_3 = \frac{1+u+u^2}{u}\\
& \mbox{for $\mathcal{S}_3$}: \quad u_1 = \frac{1+u+u^2}{u+1} , \quad u_2 = \frac{u}{u+1}  , \quad  u_3 =  u  \\
& \mbox{for $\mathcal{S}_4$}: \quad u_1 = \frac{1+u+u^2}{(u+1)u}, \quad u_2 = \frac{1}{u+1}    , \quad u_3 = \frac{1}{u}   . 
\end{align*}

\begin{remark}
Note that the some misprints occured in the formulas in \cite{ZDP} and \cite{ZGPD}: In $\mathcal{S}_{1}$, a factor of $\frac{1}{2}(u+1)$ is missing. In $\mathcal{S}_{2}$, a factor of $(u+1)$ is missing. In the formula for $\mathcal{S}_{4}$ in \cite{ZGPD}, there is a mistake in the second term in the integral and a factor of $\frac{1}{u+1}$ is missing in the front. In \cite{ZDP}, this factor of $\frac{1}{u+1}$ is also missing in the formula for $\mathcal{S}_{4}$.
\end{remark}

Before delving into the computations which will yield the above formulas, notice that in all four cases making up nontrivial resonances, namely
$$
(\epsilon_1,\epsilon_2,\epsilon_3) = (+,+,+), \;(-,-,+), \; (-,+,-), \;(+,-,-),
$$
the gradients of $F_1$ and $F_2$ are never aligned (this follows from a simple computation relying on the parametrization of the resonant manifold provided in the following subsections). Therefore, the corresponding integral is well-defined through~\eqref{formulajacobian}; but we will rather rely on~\eqref{simplecoarea} to derive explicit expressions.

\subsubsection{The case $(+,+,+)$}
\label{SSS:parametrize+++}

The first step is to consider the equations
\begin{equation}
\label{eq+++}
\left\{
\begin{array}{l}
u_3+1 = u_1 + u_2 \\
u_1^2 + u_2^2 + u_3^2 = 1.
\end{array}
\right.
\qquad u_1,u_2,u_3 \geq 0, \qquad u_2 \leq u_1.
\end{equation}
The solutions in $\mathbb{R}_+^3$ of this system can be parameterized as follows\footnote{To check this statement, first consider the system $\begin{cases} u_1 + u_2 =a \\ u_1^2 + u_2^2 = b \end{cases}$. It has solutions in $\mathbb{R}_+^2$ if and only if $a \geq 0$ and $2b \geq a^2 \geq b$. For us, $a = 1 + u_3$, $b = 1-u_3^2$, and the conditions $a \geq 0$, $2b \geq a^2 \geq b$, $u_3 \geq 0$ are equivalent to  $0 \leq u_3 \leq \frac 1 3$. In this range, $u_3$ can be parameterized by $\frac{u}{1+u+u^2}$, with $u \in [0,1]$. Finally, we check that the proposed parameterization~\eqref{uparameterization} works.}
\begin{equation}
\label{uparameterization}
u_1 = \frac{1+u}{1+u+u^2}, \quad u_2 = \frac{u(1+u)}{1+u+u^2}, \quad u_3 = \frac{u}{1+u+u^2} \quad 0 \leq u \leq 1.
\end{equation}

We start with the expression
$$
\int \mathcal{I}(\omega_1,\omega_2,\omega_3) \delta (F_1(\omega_1,\omega_2,\omega_3))  \delta (F_2(\omega_1,\omega_2,\omega_3)) \, d\omega_1 \, d\omega_2 \,d\omega_3 ,
$$
where $F_1 = \omega_1^{2} + \omega_2^{2} +\omega_3^{2}-\omega^{2}$. Changing coordinates by setting $\omega_i = \zeta_i  \omega$, the above becomes
$$
\int \mathcal{I}(\omega_1,\omega_2,\omega_3) \delta (F_1(\zeta_1,\zeta_2,\zeta_3)) \delta (F_2(\zeta_1,\zeta_2,\zeta_3)) \, d\zeta_1 \, d\zeta_2 \,d\zeta_3;
$$
using then $\delta(F_2)$ to eliminate the variable $\zeta_3=\zeta_1 + \zeta_2 -1$, we obtain
\begin{align*}
& \int \widetilde{\mathcal{I}} (\omega_1,\omega_2) \delta(\widetilde{F_1}(\zeta_1,\zeta_2)) \, d\zeta_1 \, d\zeta_2,
\end{align*}
where $\widetilde{F_1}(\zeta_1,\zeta_2) = F_1(\zeta_1,\zeta_2,\zeta_1+\zeta_2-1)$.
Here, it is possible to restrict to the case $\zeta_1 \geq \zeta_2$, since the case $\zeta_2 \geq \zeta_1$ is symmetric and will be accounted for by a factor 2. We now want to apply~\eqref{coarea2} and~\eqref{simplecoarea} with $\varphi(u) = (u_1,u_2)$. A small computation shows that
\begin{align*}
\partial_2 \widetilde{F_1} = 2 \zeta_1 + 4 \zeta_2 - 2 = \frac{2u^2 + 4u}{1+u+u^2}, \; \;\; \;
|\varphi_1'(u)| = \frac{u^2 + 2u}{1+u+u^2} \; \implies \; \left| \frac{\varphi_1'(u) }{\partial_2 \widetilde{F_1}} \right| = \frac{1}{2(1+u+u^2)}.
\end{align*}

Therefore,
\begin{equation}
\label{formulaS1}
\begin{split}
\mathcal{S}_1 (f, f,f) & = \mathcal{C}_1 (f, f,f)\\
& = 2 \omega^{4\beta + 3}  \int_0^1  \left( \frac{1+u}{1+u+u^2} \right)^{\beta+1} \left( \frac{u}{1+u+u^2} \right)^{\beta+1}  \left( \frac{u(1+u)}{1+u+u^2} \right)^{\beta+1}  \\
& \qquad \qquad \qquad \cdot ( f_1 f_2 f_3 + f_1 f_2 f_\omega - f_1 f_\omega f_3 - f_\omega f_2 f_3  ) \frac{1}{2(1+u+u^2)}  du \\
& =  \omega^{4\beta + 3}  \int_0^1 \frac{u^{2\beta +2} (u+1)^{2\beta+2}}{(1+u+u^2)^{3\beta+4}}  ( f_1 f_2 f_3 + f_1 f_2 f_\omega - f_1 f_\omega f_3 - f_\omega f_2 f_3  ) du ,
\end{split}
\end{equation}

\subsubsection{The case $(-,-, +)$}
\label{SSS:parametrize--+}

We start with the equations
\begin{equation}
\label{eq--+}
\left\{
\begin{array}{l}
u_3+1 = u_1 + u_2 \\
u_3^2 = u_1^2 + u_2^2 + 1
\end{array},
\qquad u_0,u_1,u_2 \geq 0, \qquad u_1 \leq u_2.
\right.
\end{equation}
Since the application $(u_1,u_2,u_3) \mapsto (\frac{u_1}{u_3},\frac{u_2}{u_3},\frac{1}{u_3})$ maps solutions of~\eqref{eq--+} to solutions of~\eqref{eq+++}, we see that the solutions in $\mathbb{R}_+^3$ of~\eqref{eq--+} can be parameterized by
$$
u_1 = u+1 , \quad  u_2 = \frac{u+1}{u} , \quad u_3 = \frac{1+u+u^2}{u}
$$
with $0\leq u \leq1$. As in the previous paragraph, we write
\begin{align*}
&\int \mathcal{I}(\omega_1,\omega_2,\omega_3) \delta (F_1(\omega_1,\omega_2,\omega_3))  \delta (F_2(\omega_1,\omega_2,\omega_3)) \, d\omega_1 \, d\omega_2 \,d\omega_3 \\
& \qquad\qquad\qquad\qquad \qquad\qquad\qquad\qquad\qquad = \int \widetilde{\mathcal{I}} (\omega_1,\omega_2) \delta(\widetilde{F_1}(\zeta_1,\zeta_2)) \, d\zeta_1 \, d\zeta_2,
\end{align*}
except that $F_1(\omega_1,\omega_2,\omega_3)$ is now given by $-\omega_1^2 - \omega_2^2 + \omega_3^2 + \omega^2$; and just like in the previous paragraph, we can restrict matters to the case $\zeta_1 \leq \zeta_2$ by multiplying by a factor 2. Before applying~\eqref{coarea2} and ~\eqref{simplecoarea} with $\varphi(u) = (\varphi_1(u),\varphi_2(u))$, we compute
$$
\partial_2 \widetilde{F_1} = 2\zeta_1 -2 = 2u, \quad |\varphi_1'(u)| =1 \quad \implies \quad \left| \frac{\varphi_1'(u)}{\partial_2 \widetilde{F}_1} \right| = \frac{1}{2u}.
$$
Therefore,
\begin{equation}
\label{formulaS2}
\begin{split}
& \mathcal{S}_2 (f, f, f) = \mathcal{C}_2 (f, f, f) \\
& \;\;= 2 \omega^{4 \beta +3}  \int_0^1  ( u+1 )^{\beta+1}  \left( \frac{1+u+u^2}{u} \right)^{\beta+1} \left( \frac{u+1}{u}  \right)^{\beta+1} ( f_1 f_2 f_3 + f_1 f_2 f_\omega - f_1 f_\omega f_3 - f_\omega f_2 f_3  ) \frac{1}{2u}  du \\
& \;\;= \omega^{4 \beta +3}  \int_0^1 \frac{(1+u+u^2)^{\beta+1} (u+1)^{2\beta+2}}{u^{2\beta +3} }  ( f_1 f_2 f_3 + f_1 f_2 f_\omega - f_1 f_\omega f_3 - f_\omega f_2 f_3  ) du . \\
\end{split}
\end{equation}

\subsubsection{The case $(-,+,-)$}
\label{SSS:parametrize-+-}
By symmetry between the variables $\omega_1$ and $\omega_2$, it is the same as $(+,-,-)$ below (so $\mathcal{C}_3 = \mathcal{C}_4$).

\subsubsection{The case $(+,-,-)$}
\label{SSS:parametrize+--}
First consider the equations
\begin{equation}
\label{eq+--}
\left\{
\begin{array}{l}
u_3+1 = u_1+ u_2 \\
u_1^2 = u_2^2 + u_3^2 + 1.
\end{array}
\right.
\end{equation}
Noticing that the application $(u_1,u_2,u_3) \mapsto (\frac{1}{u_1},\frac{u_3}{u_1},\frac{u_2}{u_1})$ maps solutions of~\eqref{eq+--} to solutions of~\eqref{eq+++}, we see that the solutions in $\mathbb{R}_+^3$ of~\eqref{eq+--} can be parameterized by
$$
u_1 = \frac{1+u+u^2}{u+1}, \quad  u_2 = \frac{u}{u+1}, \quad u_3 = u
$$
with $0<u<\infty$. We write now
\begin{align*}
&\int \mathcal{I}(\omega_1,\omega_2,\omega_3) \delta (F_1(\omega_1,\omega_2,\omega_3))  \delta (F_2(\omega_1,\omega_2,\omega_3)) \, d\omega_1 \, d\omega_2 \,d\omega_3 \\
& \qquad\qquad\qquad\qquad \qquad\qquad\qquad\qquad\qquad = \int \widetilde{\mathcal{I}} (\omega_1,\omega_2) \delta(\widetilde{F_1}(\zeta_1,\zeta_2)) \, d\zeta_1 \, d\zeta_2,
\end{align*}
being understood that $F_1(\omega_1,\omega_2,\omega_3)$ is now given by $\omega_1^2 - \omega_2^2 - \omega_3^2 + \omega^2$. In order to apply~\eqref{coarea2} and ~\eqref{simplecoarea} with $\varphi(u) = (\varphi_1(u),\varphi_2(u))$, we compute
$$
\partial_1 \widetilde{F_1} = -2\zeta_2 +2 = \frac{2}{u+1}, \quad |\varphi_2'(u)| = \frac{1}{(1+u)^2} \quad \implies \quad \left| \frac{\varphi_2'(u)}{\partial_1 \widetilde{F}_1} \right| = \frac{1}{2(1+u)}.
$$
This gives the formula:
\begin{equation}
\begin{split}
\mathcal{C}_4 (f, f, f)
& =   \omega^{4\beta+3} \int_0^\infty  \left( \frac{1+u+u^2}{u+1} \right)^{\beta+1} \left( \frac{u}{u+1} \right)^{\beta+1}  u^{\beta+1}  \\
& \quad \qquad \qquad \qquad \cdot ( f_1 f_2 f_3 + f_1 f_2 f_\omega - f_1 f_\omega f_3 - f_\omega f_2 f_3  ) \frac{1}{2(u+1)}  du \\
& = \frac{1}{2} \omega^{4\beta+3}  \int_0^\infty \frac{(1+u+u^2)^{\beta+1}u^{2\beta+2}}{(u+1)^{2\beta+3}}  ( f_1 f_2 f_3 + f_1 f_2 f_\omega - f_1 f_\omega f_3 - f_\omega f_2 f_3  ) \, du ,  \\
\end{split}
\end{equation}
where
$$
\omega_1 = \omega u_1, \qquad \omega_2 = \omega u_2, \qquad \omega_3 = \omega u_3.
$$
Splitting the above integral into $\int_0^1$ and $\int_1^\infty$ and making a change of variables $u \rightarrow 1/u$ for the $\int_1^\infty$ integral, we obtain
\begin{equation*}
\begin{split}
\mathcal{C}_3 (f, f, f) + \mathcal{C}_4 (f, f, f) = \mathcal{S}_3 (f, f, f) + \mathcal{S}_4 (f, f, f) 
\end{split}
\end{equation*}
where $\mathcal{S}_3$ and $\mathcal{S}_4$ have the desired form.
\subsection{A more symmetrical representation formula}
In this section, we derive the following representation formula for the collision operator:
\begin{equation} \label{collision_rep}
\mathcal{C}(n) (\omega) = \omega^{3+4\beta} \int_0 ^1 \sum_{j,k=1} ^4 (-1)^{j+k} \frac{1}{1+u+u^2}  v_k^{-1} \left ( \prod_{i=1} ^4 \frac{v_i}{v_k} \right )^{\beta+1} \frac{\prod_{i=1} ^4 n_{i,k} }{n_{j,k} } \,du,
\end{equation}
where 
\begin{equation}
\label{defvi}
v_1 = \frac{1+u}{1+u+u^2}, \quad v_2 = 1, \quad v_3 = \frac{u(1+u)}{1+u+u^2}, \quad v_4 = \frac{u}{1+u+u^2},
\end{equation}
$v_{i,k} = v_i/v_k$, and $n_{i,k} = n (\omega v_{i,k})$. 

This formula can be established from the decomposition~\eqref{ZGPDparameterization}, by observing that $\mathcal{S}_1$, $\mathcal{S}_2$, $\mathcal{S}_3$, $\mathcal{S}_4$ correspond to the summands for $k=2,4,1,3$ in~\eqref{collision_rep} respectively. The advantage of this formulation, besides its compactness, is that it encapsulates two different cancellations: the alternating sum over $j$ corresponds to the classical cancellation between gain and loss terms in each of the terms $\mathcal{S}_i$, while the alternating sum over $k$ expresses a cancellation arising between the different terms $\mathcal{S}_i$.

It turns out that the combination of cancellation in $j$ and $k$ is responsible for the gain of regularity. Also, we remark that the even symmetry assumption $n(-k) = n(k)$ was crucial in obtaining $(-1)^k$ cancellation. It would be interesting if one could obtain the cancellation structure without even symmetry.

\section{The collision operator in weighted $L^p$ spaces}

\label{sectionweightedLp}

\subsection{Boundedness in weighted $L^\infty$}
For any $f = f(\omega)$, denote 
\begin{equation}
\| f \|_{L^\infty_{\theta, \gamma}} := \| m(\omega) f \|_{L^\infty} ,
\end{equation}
where
\[ m(\omega)=
\begin{cases} 
|\omega|^{-\theta} & | \omega | \leq 1 \\
|\omega|^\gamma & | \omega | > 1 
\end{cases}
\]
and $\theta$, $\gamma$ are constants.

\begin{lemma} \label{lemma1}
If $-1/2 \leq \beta \leq 0$, $\gamma > 2\beta + 2$ and $\theta > -2\beta -1$,
\begin{equation}
\|\mathcal{F} (f, g, h)\|_{L^\infty_{\theta, \gamma}} \lesssim \|f\|_{L^\infty_{\theta, \gamma}} \|g\|_{L^\infty_{\theta, \gamma}} \|h\|_{L^\infty_{\theta, \gamma}} .
\end{equation}
\end{lemma} 

\begin{remark} The ranges of the parameters $\beta$, $\theta$ and $\gamma$ in this statement are optimal if one ignores the cancellations in the collision term, as can be seen from the proof. The critical space given the scaling of the equation would be $\theta = - \gamma = -2\beta - \frac{3}{2}$.
\end{remark}

\begin{proof}
The proof of this result will not take advantage of any cancellation arising from the alternating signs in the collision operators. Accordingly, we will denote $\mathcal{F}$ for the collision operator where all signs have been changed to $+$. The decomposition~\eqref{ZGPDparameterization} becomes $\mathcal{F} = \mathcal{F}_1 + \mathcal{F}_2 + \mathcal{F}_3 + \mathcal{F}_4$.

It suffices to prove that 
\begin{equation}
m (\omega) |\mathcal{F} (m (\omega)^{-1}, m (\omega)^{-1}, m (\omega)^{-1} )| \lesssim 1.
\end{equation}

\medskip

\noindent \underline{The term $\mathcal{F}_1$.} On the one hand, if $\omega \leq 1$, it follows from~\eqref{formulaS1} that
\begin{equation}
\label{F1<1}
\begin{split}
&  \omega^{-\theta} \mathcal{F}_1 (m (\omega)^{-1}, m (\omega)^{-1}, m (\omega)^{-1} ) \lesssim  \omega^{4 \beta +3-\theta} \int_0^1 u^{2\beta +2} ( u^{2\theta} \omega^{3\theta} + u^{\theta} \omega^{3\theta}   ) du \lesssim \omega^{4 \beta +3 + 2\theta}
\end{split}
\end{equation}
if 
\begin{equation}
\label{condition1}
2 \beta +3 +  \theta > 0 \quad \mbox{and} \quad 2\beta + 3 + 2\theta >0.
\end{equation}
The right-hand side of~\eqref{F1<1} is $\lesssim 1$ if
\begin{equation}
\label{condition2}
4 \beta +3 +2 \theta \geq 0.
\end{equation}

On the other hand, if $\omega > 1$, 
\begin{equation}
\label{F1>1}
\begin{split}
&   \omega^{\gamma} \mathcal{F}_1  (m (\omega)^{-1}, m (\omega)^{-1}, m (\omega)^{-1} ) \\
& \qquad  \lesssim \omega^{4 \beta +3 + \gamma}  \left[  \int_{1/\omega}^1  u^{2\beta +2} ( u^{-\gamma} \omega^{-3 \gamma}  + u^{-2 \gamma} \omega^{-3 \gamma}  ) du \right.+ \left.  \int_0^{1/\omega}  u^{2\beta +2} ( u^{2\theta} \omega^{2\theta-\gamma} + u^{\theta} \omega^{\theta-2 \gamma}   )  du \right]  \\
& \qquad \sim  \omega^{4 \beta +3 - 2\gamma} +   \omega^{2\beta- \gamma} +  \omega^{2\beta  } .
\end{split}
\end{equation}
if
\begin{equation}
\label{condition3}
2\beta + 3 + \theta  >0 \quad \mbox{and} \quad \quad 2\beta + 3 + 2 \theta >0.
\end{equation}

The right-hand side of~\eqref{F1>1} is $\lesssim 1$ provided 
\begin{equation}
\label{condition4}
4 \beta +3 - 2\gamma \leq 0, \quad 2 \beta - \gamma \leq 0 \quad \mbox{and} \quad \beta \leq 0. 
\end{equation}

\medskip

\noindent \underline{The term $\mathcal{F}_2$.} On the one hand, if $\omega < 1$, 
\begin{equation}
\label{F2<1}
\begin{split}
&   \omega^{-\theta} \mathcal{F}_2  (m (\omega)^{-1}, m (\omega)^{-1}, m (\omega)^{-1} ) \\
& \qquad  \lesssim \omega^{4 \beta +3 - \theta}  \left[  \int_{0}^\omega  u^{-2\beta -3}( \omega^{2\theta -\gamma} u^{\gamma}  + \omega^{\theta -2 \gamma} u^{2\gamma}  ) du \right. + \left.  \int_\omega^1 u^{-2\beta -3} ( \omega^{3\theta} u^{-2\theta} + \omega^{3\theta} u^{-\theta}  )  du \right]  \\
& \qquad \sim \omega^{4\beta + 3 +2 \theta} + \omega^{2\beta + 1 + \theta} + \omega^{2\beta +1} 
\end{split}
\end{equation}
if
\begin{equation}
\label{condition5}
2\beta + 2 - \gamma < 0 \quad \mbox{and} \quad 2\beta + 2 - 2\gamma < 0.
\end{equation}

The right-hand side of~\eqref{F2<1} is $\lesssim 1$ provided 
\begin{equation}
\label{condition6}
4\beta + 3 +2 \theta \geq 0, \quad 2\beta + 1 + \theta \geq 0, \quad \mbox{and} \quad 2\beta+1 \geq 0.\end{equation}

If $\omega > 1$, we have
\begin{equation}
\label{F2>1}
\begin{split}
& \omega^{\gamma} \mathcal{F}_2  (m (\omega)^{-1}, m (\omega)^{-1}, m (\omega)^{-1} ) \lesssim  \omega^{4 \beta +3 + \gamma} \int_0^1 u^{-2 \beta -3} (\omega^{-3\gamma} u^\gamma + \omega^{-3\gamma} u^{2\gamma}) \,du \lesssim \omega^{4\beta + 3 - 2\gamma}
\end{split}
\end{equation}
if
\begin{equation}
\label{condition7}
2 \beta + 2 - \gamma < 0 \quad \mbox{and} \quad 2 \beta + 2 - 2 \gamma < 0.
\end{equation}
The right-hand side of~\eqref{F2>1} is $\lesssim 1$ if
\begin{equation}
\label{condition8}
4\beta + 3 - 2\gamma \leq 0.
\end{equation}

\medskip
\noindent
\underline{The term $\mathcal{F}_3$} It can be bounded exactly like $\mathcal{F}_1$.

\medskip
\noindent
\underline{The term $\mathcal{F}_4$} It can be bounded exactly like $\mathcal{F}_2$.

\medskip
\noindent
\underline{Conclusion} The desired statement follows from combining the conditions~\eqref{condition1}, \eqref{condition2}, \eqref{condition3}, \eqref{condition4}, \eqref{condition5}, \eqref{condition6}, \eqref{condition7} and \eqref{condition8}.
\end{proof}

\subsection{Unboundedness in weighted $L^p$ spaces for $p<3$}

\begin{lemma} \label{lemma2}
For any $p < 3$, there exists data $( f^\epsilon )$ in $L^\infty(\frac{1}{10},2)$ such that
$$
\| f^\epsilon \|_{L^p} = O(1) \quad \mbox{while} \quad \| \mathcal{C}(f^\epsilon) \|_{L^p(\frac{1}{2},\frac{3}{2})} \to \infty \quad \mbox{as $\epsilon \to 0$}.
$$
\end{lemma}

\begin{proof}
We start with the following combinatorial fact: consider the system of equations
\begin{align*}
& x_1 + x_2 -x_3 - x_4 = 0 \\
& x_1^2 + x_2^2 + x_3^2 - x_4^2 = 0.
\end{align*}
The solutions of this system of equations such that, for some $\{ i,j,k \} \subset \{ 1,2,3,4 \}$, there holds $\{ x_i,x_j,x_k \} \subset \{1/3,2/3,1\}$, reduce to
$$
(x_1,x_2,x_3,x_4) = (2/3,2/3,1/3,1).
$$

We now consider the data
$$
f^\epsilon (\omega) = \epsilon^{-\frac{2}{p}} \mathbf{1}(\epsilon^{-2}(\omega - 1/3)) + \epsilon^{-\frac{1}{p}} \mathbf{1}(\epsilon^{-1}(\omega - 2/3)) + \epsilon^{-\frac{2}{p}} \mathbf{1}(\epsilon^{-2}(\omega - 1)),
$$
which is such that
$$
\| f^\epsilon \|_{L^p} \sim 1.
$$

Our aim is now to estimate the size of $\mathcal{C}(f^\epsilon)(\omega)$ if $|\omega - 1| \ll \epsilon^2$. Recall from the formula~\eqref{MMT-kwe} that the collision operator $\mathcal{C}$ can be split into $\mathcal{C}_1 + \mathcal{C}_2 + \mathcal{C}_3 + \mathcal{C}_4$.

\medskip

\noindent \underline{The contribution of $\mathcal{C}_2(f^\epsilon)$.} The kernel of this operator is supported on the zero set of $\omega_1 + \omega_2 - \omega_3 - \omega$ and $\omega_1^2 + \omega_2^2 - \omega_3^2 + \omega^2$. By the combinatorial fact above, the term $\mathcal{C}_2(f^\epsilon)$ is localized close to $1/3$, and therefore
$$
\mathcal{C}_2(f^\epsilon)(\omega) = 0 \qquad \mbox{if $|\omega - 1| \ll 1$}.
$$

\medskip

\noindent \underline{The contribution of $\mathcal{C}_3(f^\epsilon) = \mathcal{C}_4(f^\epsilon)$.} The kernel of this operator is supported on the zero set of $\omega_1 + \omega_2 - \omega_3 - \omega$ and $\omega_1^2 - \omega_2^2 + \omega_3^2 + \omega^2$. By the combinatorial fact above, the term $\mathcal{C}_3(f^\epsilon)$ is localized close to $2/3$, and therefore
$$
\mathcal{C}_3(f^\epsilon)(\omega) = \mathcal{C}_4(f^\epsilon)(\omega) = 0 \qquad \mbox{if $|\omega - 1| \ll 1$}.
$$

\medskip

\noindent \underline{The contribution of $\mathcal{C}_1(f^\epsilon) $.} The kernel of this operator is supported on the zero set of $\omega_1 + \omega_2 - \omega_3 - \omega_4$ and $\omega_1^2 + \omega_2^2 + \omega_3^2 - \omega^2$. By the combinatorial fact above, only frequencies close to $(\omega_1,\omega_2,\omega_3,\omega) = (2/3,2/3,1/3,1)$ have a nonzero contribution to $\mathcal{C}_1(f^\epsilon)$. 

We now recall the formula~\eqref{formulaS1} for $\mathcal{C}_1$:
\begin{align*}
\mathcal{C}_1(f^\epsilon) 
& = \omega^{4\beta + 3}  \int_0^1 \frac{u^{2\beta +2} (u+1)^{2\beta+2}}{(1+u+u^2)^{3\beta+4}}  ( f_1 f_2 f_3 + f_1 f_2 f_\omega - f_1 f_\omega f_3 - f_\omega f_2 f_3  ) \, du\\
& = \mathcal{C}_1^1 (f^\epsilon) +  \mathcal{C}_1^2 (f^\epsilon) +  \mathcal{C}_1^3 (f^\epsilon) +  \mathcal{C}_1^4 (f^\epsilon)
\end{align*}
where
\begin{align*}
& \omega_1 = \omega u_1, \qquad \omega_2 = \omega u_2, \qquad \omega_3 = \omega u_3 \\
& u_1 = \frac{1+u}{1+u+u^2}, \quad u_2 = \frac{u(1+u)}{1+u+u^2}, \quad \mbox{and} \quad u_3 = \frac{u}{1+u+u^2}.
\end{align*}

\medskip

\noindent \underline{The contribution of $\mathcal{C}_1^1(f^\epsilon) $.} It is given by
\begin{align*}
\mathcal{C}_1^1(f^\epsilon) = \epsilon^{-4/p} \omega^{4\beta + 3} &  \int_0^1 \frac{u^{2\beta +2} (u+1)^{2\beta+2}}{(1+u+u^2)^{3\beta+4}}  \\
& \qquad \qquad \mathbf{1}(\epsilon^{-1}(\omega_1 - 2/3)) \mathbf{1}(\epsilon^{-1}(\omega_2 - 2/3))  \mathbf{1}(\epsilon^{-2}(\omega_3 - 1/3)) \,du
\end{align*}
Notice that
$$
u_1(1) = \frac{2}{3} \quad \mbox{and} \quad \left. \frac{d}{du} u_1(u) \right|_{u=1} = - \frac{1}{3} \neq 0.
$$
Therefore, if $\omega$ is close to $1$, the function $u \mapsto \mathbf{1}(\epsilon^{-1}(\omega_1 - 2/3))$ is nonzero on a set of a length $O(\epsilon)$. Therefore,
$$
| \mathcal{C}_1^1 (f^\epsilon)(\omega)| \lesssim \epsilon^{1- \frac{4}{p}}.
$$

\medskip

\noindent \underline{The contribution of $\mathcal{C}_1^2(f^\epsilon) $.} It can be estimated just like $\mathcal{C}_1^1$: we find, for $\omega$ close to $1$,
$$
| \mathcal{C}_1^2 (f^\epsilon)(\omega)| \lesssim \epsilon^{1- \frac{4}{p}}.
$$

\medskip

\noindent \underline{The contribution of $\mathcal{C}_1^3(f^\epsilon) $.} 
It is given by
\begin{align*}
\mathcal{C}_1^3(f^\epsilon) = \epsilon^{-5/p} \omega^{4\beta + 3} &  \int_0^1 \frac{u^{2\beta +2} (u+1)^{2\beta+2}}{(1+u+u^2)^{3\beta+4}}  \\
& \qquad \qquad \mathbf{1}(\epsilon^{-1}(\omega_1 - 2/3)) \mathbf{1}(\epsilon^{-2}(\omega_3 - 1/3))  \mathbf{1}(\epsilon^{-2}(\omega - 1)) \,du
\end{align*}

The crucial point is that
$$
u_3(1) = \frac{1}{3} \quad \mbox{and} \quad \left. \frac{d}{du} u_3(u) \right|_{u=1} =0.
$$
Therefore, if $|\omega - 1| \ll \epsilon^2$ and $|u -1 | \ll \epsilon$, there holds
$$
\mathbf{1}(\epsilon^{-2}(\omega - 1)) = \mathbf{1}(\epsilon^{-1}(\omega_1 - 2/3)) =  \mathbf{1}(\epsilon^{-2}(\omega_3 - 1/3)) = 1.
$$
As a result, we can estimate from below, if $|\omega - 1| \ll \epsilon^2$,
$$
|\mathcal{C}^3_1(f^\epsilon)(\omega)| \gtrsim \epsilon^{1 - \frac{5}{p}}.
$$

\medskip

\noindent \underline{The contribution of $\mathcal{C}_1^4(f^\epsilon) $.} By the same argument,
$$
|\mathcal{C}^4_1(f^\epsilon)(\omega)| \gtrsim \epsilon^{1 - \frac{5}{p}}
$$
(but this lower bound is not even necessary for the rest of the argument, since $\mathcal{C}^4_1$ has the same sign as $\mathcal{C}^3_1$).

\medskip

\noindent \underline{Conclusion} Gathering the contributions of all the terms above, we find that, if $|\omega - 1| \ll \epsilon^2$,
$$
| \mathcal{C}(f^\epsilon)(\omega)| \gtrsim \epsilon^{1 - \frac{5}{p}}.
$$
Therefore,
$$
\|\mathcal{C}(f^\epsilon) \|_{L^p} \gtrsim \epsilon^{1 - \frac{3}{p}},
$$
from which the desired result follows.
\end{proof}

\section{Local well-posedness near scale invariant power data}

\label{sectionscaling}

In this section, we establish a priori estimate and local existence of solutions near scale-invariant power spectrum. We set $\beta \in (-1, 1)$, $\gamma = 2 \beta + \frac{3}{2}$, and let 
\begin{equation}
N(\omega) = \omega ^\gamma n(\omega).
\end{equation}
In particular, $\gamma$ for $\beta = -\frac{3}{4}, -\frac{1}{4}$ correspond to the Rayleigh-Jean spectrum, and $\gamma$ for $\beta = \frac{1}{4}, \frac{3}{4}$ correspond to the Kolmogorov-Zhakarov spectrum. 

From \eqref{collision_rep}, the equation for $N$ reads the following:
\begin{equation} \label{Neqn}
\partial_t N(\omega) = \int_0 ^1  \frac{\left (\prod_{i=1} ^4 v_i \right )^{-\beta - \frac{1}{2} } }{1+u+u^2} \sum_{j,k=1} ^4  (-1)^{j+k} v_j ^{2\beta+ \frac{3}{2} } v_k^{2\beta - \frac{1}{2} } \frac{\prod_{i=1} ^4 N_{i,k} }{N_{j,k} } du =: \overline{\mathcal{C}} (N) (\omega),
\end{equation}
where $N_{i,k} = N(\omega v_i/v_k )$. Also, we recall
\begin{equation}
DN (\omega) = \omega \frac{d}{d\omega} N(\omega).
\end{equation}
One has the following useful formula:
\begin{equation}
D(N_{i,k}) = \omega \frac{d}{d\omega}  ( N_{i,k} ) = \omega v_i/v_k N' (\omega v_i/v_k) = (DN) _{i,k}
\end{equation}
and therefore the expression $DN_{i,k}$ is not ambiguous.

\medskip

\noindent \underline{A priori assumptions in Section \ref{Nonnegativity}, \ref{Linfty}, and \ref{Lipschitz}.} We will assume the following:
\begin{enumerate}
\item[(A1)] $p_0$ is a positive integer, and $p_0$ and $\beta$ satisfy 
\begin{equation} \label{betaanddata}
\beta \in (-1, 1), \qquad p_0 > \max \left ( \frac{1}{4(1-\beta)}, \frac{1}{4(1+\beta) } \right ).
\end{equation}

\smallskip
\item[(A2)] $N \in L^\infty (0, T; X)$ is a strong solution of \eqref{Neqn} with $N(0) = N_0 \in X$, that is, for almost every $t \in (0, T)$ $\partial_t N \in L^\infty (0, T; L^\infty((0, \infty ) ))$, $D \partial_t N \in L^1 (0, T; L^{2p_0} (0, \infty ) \cap L^\infty (0, \infty )  )$, and \eqref{Neqn} is satisfied for almost every $(t, \omega)$, 

\smallskip
\item[(A3)] $\inf_{\omega \in (0, \infty ) } N_0 (\omega) > 0$, and

\smallskip
\item[(A4)] We have
\begin{equation}
\int_0 ^{T} \int_0 ^\infty \int_{-1/3} ^{1 }  | DN(\omega) - DN(\omega(1+u) )|^2 \left ( \mathbf{1}_{[-1/3, 0]} (u) u^{-(\beta+1)} + \mathbf{1}_{[0, 1]} (u) u^{2\beta} \right )  du d\omega dt < \infty.
\end{equation} 
\end{enumerate}

\subsection{A priori estimates: Positivity of $N$} \label{Nonnegativity}

We first establish the non-negativity of $N$: suppose that $N_0 (\omega) \ge 0$ for every $\omega>0$ and $t$ is the first time such that there is a $\omega>0$ with $N(t, \omega) = 0$. Then
\begin{equation}
\partial_t N(t, \omega) = \overline{\mathcal{C}}(N) (t, \omega) = \int_0 ^1 \frac{\left (\prod_{i=1} ^4 v_i \right )^{-\beta - \frac{1}{2} } }{1+u+u^2} \sum_{j,k=1} ^4 \delta_{j} ^{k}  (-1)^{j+k} v_j ^{2\beta+ \frac{3}{2} } v_k^{2\beta - \frac{1}{2} } \frac{\prod_{i=1} ^4 N_{i,k} }{N_{j,k} } du
\end{equation} 
since these are the only terms which does not contain $N_{k,k} = N(\omega)$. Since $N_{i,k} \ge 0$ for $i\ne k$, $N(t', \omega) \ge 0$ for all $\omega \in (0, \infty)$ and $t'>t$ with $t'$ sufficiently close to $t$. Thus, $N$ remains nonnegative. Moreover, if we write
\begin{equation}
\partial_t N^{-1} (t, \omega) = - (N^{-1})^2 \overline{C} (N) (t, \omega).
\end{equation}
We see, from Section \ref{Linfty}, that
\begin{equation}
\| \partial_t N^{-1} \|_{L^\infty} \lesssim \| N^{-1} \|_{L^\infty}^2 \| N\|_{L^\infty}^2 (\| N\|_{L^\infty} +  \| DN \|_{L^\infty}),
\end{equation}
and in the end we have $\| N\|_{L^\infty}^2 (\| N\|_{L^\infty} +  \| DN \|_{L^\infty}) \lesssim 1$. This means that $N^{-1}$ is bounded above for a positive time, or $N$ is bounded below.
To summarize, we have the following:
\begin{lemma} \label{positivitylemma}
There exists a $T_p = T(N_0, \beta, p_0) >0$ and a constant $C(N_0, \beta, p_0)$ such that
$$ \inf_{t \in (0, T_p), \omega \in (0, \infty ) } N(t, \omega) \ge C(N_0, \beta, p_0) > 0. $$
\end{lemma}

\subsection{A priori estimates: $L^\infty$ estimate} \label{Linfty}
For $L^\infty$ estimate, we use the cancellation coming from gain-loss relation, and thus this is a losing estimate. In the next section, we will establish a non-losing estimate. 
We start from the following simple but useful observation: for $u \in [0,1]$, 
\begin{equation} \label{visizes}
v_1 \sim v_2 \sim 1, \qquad v_3 \sim v_4 \sim u, \qquad v_1 - v_2 = v_4 - v_3 \sim u^2.
\end{equation}

Breaking down $\overline{\mathcal{C}} (N)$ into 4 parts, we obtain
\begin{equation}
\begin{split}
\overline{\mathcal{C}}(N) (\omega) &= \int_0 ^1 W(u) \sum_{k=1} ^4 (-1)^k v_k ^{2\beta - \frac{1}{2} } \times \\
&\left [  (v_2 ^{2\beta + 3/2} - v_1^{2\beta + 3/2}) N_{2,k}N_{3,k}N_{4,k}  + v_2^{2\beta+3/2} (N_{1,k} - N_{2,k})N_{3,k}N_{4,k} \right. \\
&\left.+ (v_4^{2\beta+3/2} - v_3^{2\beta+3/2} ) N_{1,k}N_{2,k}N_{4,k} + v_4^{2\beta+3/2} (N_{3,k} - N_{4,k} ) N_{1,k} N_{2,k}  \right] du,
\end{split}
\end{equation}
where
\begin{equation} \label{weight}
W(u) = \frac{\left (\prod_{i=1} ^4 v_i \right )^{-\beta - \frac{1}{2} } }{1+u+u^2} \sim u^{-2\beta -1}.
\end{equation}
Then, by the Newton-Leibniz formula and the definition of $DN$, we have
\begin{equation}
N_{1,k} - N_{2,k} = \int_0 ^1 DN \left ( \omega \frac{\lambda v_1 + (1-\lambda) v_2}{v_k} \right ) \frac{v_1 - v_2}{\lambda v_1 + (1-\lambda) v_2 } d\lambda
\end{equation}
and consequently by \eqref{visizes}, 
\begin{equation} \label{Ndiffsize1}
|N_{1,k} - N_{2,k} | \lesssim \| DN \|_{L^\infty} \frac{|v_1 - v_2|}{\min (v_1, v_2) } \lesssim u^2 \|DN \|_{L^\infty}.
\end{equation}
Similarly,
\begin{equation} \label{Ndiffsize2}
|N_{3,k} - N_{4,k} | \lesssim \| DN \|_{L^\infty} \frac{|v_3 - v_4|}{\min (v_3, v_4 ) } \lesssim u \| DN \|_{L^\infty}
\end{equation}
and 
\begin{equation} \label{vdiffsize}
|v_2^{2\beta + 3/2} - v_1^{2\beta+3/2} | \lesssim u^2, \qquad |v_4^{2\beta+3/2} - v_3^{2\beta+3/2} | \lesssim u^{2\beta+5/2}.
\end{equation}
Putting all these together, we have
\begin{equation} \label{collisionpointwise}
\begin{split}
\left | \overline{\mathcal{C}} (N) (\omega) \right | &\lesssim \int_0 ^1 u^{-2\beta - 1} ( 1 + u^{2\beta - 1/2} ) \\
&\qquad \qquad \left [ (u^2 + u^{2\beta+5/2} ) \| N \|_{L^\infty}^3 + (u^2 + u^{2\beta+5/2} )\|N\|_{L^\infty}^2 \| DN \|_{L^\infty} \right ] du \\
&\lesssim \int_0^1 u^{-2\beta - 1} (1+u^{2\beta-1/2} )(u^2 + u^{2\beta+5/2} ) du \|N\|_{L^\infty}^2 (\|N\|_{L^\infty} + \|DN \|_{L^\infty} ),
\end{split}
\end{equation}
and the right-hand side integral converges if and only if $\beta \in (-1, 1)$, as the integrand is $u^{-2\beta + 1 } + u^{3/2} + u^{1/2} + u^{2\beta + 1}$. Therefore, we have proved the following:
\begin{lemma}
We have
\begin{equation} \label{LinftyN}
\frac{d}{dt} \| N \|_{L^\infty} \le C \| N \|_{L^\infty}^2 (\| N \|_{L^\infty} + \| DN \|_{L^\infty} ).
\end{equation}
\end{lemma}
\subsection{A priori estimates: $\dot{W}^{1,p} \cap W^{1, \infty}$ estimate and regularity gain} \label{Lipschitz}
In this section, we will obtain a local-in-time $L^{2p}$ estimate for $DN$ which is uniform in $p$ for every positive integer $p\ge p_0$, thereby giving an $L^\infty$ estimate for $DN$. 

\medskip

\noindent \underline{Splitting the integral} Let $p$ be a positive integer. We have
\begin{equation}
\begin{split}
\|D N \|_{L^{2p} }^{2p-1}\frac{d}{dt} \|D N \|_{L^{2p} }  = \int_0 ^\infty (DN)^{2p-1} (\omega) D ( \overline{C} (N) ) (\omega) d \omega = \sum_{i=1}  ^6 I_i,
\end{split}
\end{equation}
where
\begin{equation}
\begin{split}
I_1 &= \int_0 ^\infty \int_0 ^1 (DN)^{2p-1} (\omega) W(u) \sum_{k=1} ^4 (-1)^k v_k^{2\beta - 1/2}  (v_2 ^{2\beta+3/2} - v_1^{2\beta+3/2} ) D(N_{2,k} N_{3,k} N_{4,k} ) \, du\, d\omega, \\
I_4 &= \int_0 ^\infty \int_0 ^1 (DN)^{2p-1} (\omega) W(u) \sum_{k=1} ^4 (-1)^k v_k^{2\beta - 1/2}  (v_4 ^{2\beta+3/2} - v_3^{2\beta+3/2} ) D(N_{1,k} N_{2,k} N_{4,k} ) \, du \, d\omega,
\end{split}
\end{equation}
\begin{equation}
\begin{split}
I_2 &= \int_0 ^\infty \int_0 ^1 (DN)^{2p-1} (\omega) W(u) \sum_{k=1} ^4 (-1)^k v_k^{2\beta - 1/2} v_2^{2\beta+3/2} (DN_{1,k} - DN_{2,k}) N_{3,k} N_{4,k} \, du \, d\omega, \\
I_5 &= \int_0 ^\infty \int_0 ^1 (DN)^{2p-1} (\omega) W(u) \sum_{k=1} ^4 (-1)^k v_k^{2\beta - 1/2} v_4^{2\beta+3/2} (DN_{3,k} - DN_{4,k} ) N_{1,k} N_{2,k} \, du \, d\omega,
\end{split}
\end{equation}
and
\begin{equation}
\begin{split}
I_3 &=\int_0 ^\infty \int_0 ^1 (DN)^{2p-1} (\omega) W(u) \sum_{k=1} ^4 (-1)^k v_k^{2\beta - 1/2} v_2^{2\beta+3/2} (N_{1,k} - N_{2,k} ) D(N_{3,k} N_{4,k}) \, du \, d\omega, \\
I_6 &= \int_0 ^\infty \int_0 ^1 (DN)^{2p-1} (\omega) W(u) \sum_{k=1} ^4 (-1)^k v_k^{2\beta - 1/2} v_4^{2\beta+3/2} (N_{3,k} - N_{4,k} ) D(N_{1,k} N_{2,k} ) \, du \, d\omega.
\end{split}
\end{equation}
The morale of the splitting is the following: since we aim not to lose derivatives, we take a closer look at the terms involving cancellation of $DN$s, that is, $I_2, I_5$. For the other terms, $I_1, I_3, I_4$, and $I_6$, the usual cancellation between $N$s can be used. For $I_2$ and $I_5$, it turns out that the most dangerous terms are those involving $DN_{1,1} - DN_{2,1}$ and $DN_{1,2} - DN_{2,2}$ for $I_2$, and $DN_{3,3} - DN_{4,3}$ and $DN_{3,4} - DN_{4,4}$ for $I_5$: the terms involving $DN_{1,3} - DN_{2,3}$, $DN_{1,4} - DN_{2,4}$ for $I_2$ and its counterpart in $I_5$ can be treated using integration by parts. To elaborate this, we see that 
$$DN_{1,3} - DN_{2,3} = \int_0 ^1 (DN)' \left (\frac{ \lambda v_1 + (1-\lambda) v_2 }{v_3} \right ) d\lambda \frac{v_1 - v_2} {v_3}.$$
However, we see that 
$$(DN)' \left (\frac{ \lambda v_1 + (1-\lambda) v_2 }{v_3} \right ) = \left ( \frac{d}{du} (DN) \right )  \left (\frac{ \lambda v_1 + (1-\lambda) v_2 }{v_3} \right ) \left [ \frac{d}{du} \left (\frac{ \lambda v_1 + (1-\lambda) v_2 }{v_3} \right ) \right ]^{-1},$$
and the term $\frac{v_1 - v_2} {v_3}  \left [ \frac{d}{du} \left (\frac{ \lambda v_1 + (1-\lambda) v_2 }{v_3} \right ) \right ]^{-1} \sim u^3$. Thus, integration by parts with respect to $u$ can resolve the loss of derivatives issue, while we can extract enough cancellation to keep the integral finite.

The remaining terms have good signs, which gives us an analog of parabolicity. 

\medskip

\noindent \underline{Estimate of $I_1$ and $I_4$} By H\"older's inequality, \eqref{weight}, \eqref{vdiffsize} and Minkowski's integral inequality, 
\begin{equation}
\begin{split}
|I_1| &\lesssim \| DN \|_{L^{2p}} ^{2p-1} \left [ \int_0 ^\infty \left | \int_0 ^1 W(u) \sum_{k=1}^4 (-1)^k  v_k^{2\beta - 1/2} (v_2^{2\beta + \frac{3}{2}} - v_1^{2\beta + \frac{3}{2}}) D(N_{2,k} N_{3,k} N_{4,k} )  \,du\right |^{2p} d\omega \right ]^{1/2p} \\
&\lesssim \| DN \|_{L^{2p}} ^{2p-1} \sum_{i=2} ^4 \sum_{k=1} ^4 \int_0 ^1 u^{-2\beta +1} v_k^{2\beta - 1/2} \| DN_{i,k} \|_{L^{2p}  } \| N \|_{L^\infty} ^2 \,du
\end{split}
\end{equation}
where $\| DN_{i,k} \|_{L^{2p}  }$ is an integral with respect to $\omega$, with $u$ as a parameter, and the implicit constant is uniform in $p$. By change of variables, we have
\begin{equation}
\| DN_{i,k} \|_{L^{2p} } = \| DN \|_{L^{2p}} (v_{k,i})^{1/2p}.
\end{equation}
Therefore, we have
\begin{equation}
|I_1| \lesssim \| DN\|_{L^{2p}}^{2p} \| N \|_{L^\infty}^2 \sum_{i=2} ^4 \sum_{k=1} ^4 \int_0 ^1 u^{-2\beta + 1} v_k^{2\beta - 1/2} (v_{k,i} )^{1/2p} \, du.
\end{equation}
Next, note that $v_{k,i} \lesssim u^{-1}$ for any $k,i$, and thus $(v_{k,i})^{1/2p} \lesssim (u^{-1} )^{1/2p} \le (u^{-1} )^{1/2p_0}$. Thus, if $p_0$ satisfies $-2\beta + 1 - 1/2p_0 > -1$, that is,
\begin{equation} \label{p0cond1}
p_0 > \frac{1}{4(1-\beta) },
\end{equation}
then
\begin{equation}
|I_1| \lesssim \| DN\|_{L^{2p} }^{2p} \| N \|_{L^\infty}^2 \int_0 ^1 u^{-2\beta+1 - \frac{1}{2p_0} } (1+u^{2\beta - 1/2} ) du \le C (p_0, \beta)\| DN\|_{L^{2p} }^{2p} \| N \|_{L^\infty}^2,
\end{equation}
where $C(p_0, \beta)$ depends on $p_0$ and $\beta$ but not on $p$ itself. $I_4$ is estimated in the exactly same manner: the only difference comes from the equivalent $|v_4^{2\beta+3/2} - v_3^{2\beta + 3/2} | \sim u^{2\beta+5/2}$ instead of $u^2$ (see  \eqref{vdiffsize}), which leads to
\begin{equation}
\begin{split}
|I_4| &\lesssim \| DN \|_{L^{2p} }^{2p} \| N \|_{L^\infty} ^2  \sum_{i=1} ^3 \sum_{k=1} ^4 \int_0 ^1 u^{3/2} v_k^{2\beta - 1/2} (v_{k,i } )^{1/2p} \, du \\
&\lesssim \| DN \|_{L^{2p} }^{2p} \| N \|_{L^\infty} ^2 \int_0^1 u^{3/2 - \frac{1}{2p_0} } (1+u^{2\beta - 1/2} ) \, du \le C(p_0, \beta) \| DN \|_{L^{2p} }^{2p} \| N \|_{L^\infty} ^2
\end{split}
\end{equation}
provided that $2\beta + 1 - \frac{1}{2p_0} > -1$, that is,
\begin{equation} \label{p0cond2}
p_0 > \frac{1}{4(1+\beta) }.
\end{equation}

\medskip

\noindent \underline{Estimate of $I_3$ and $I_6$.} The idea is the same as for $I_1$ and $I_4$: for $I_3$, we use \eqref{Ndiffsize1} and for $I_6$ we use \eqref{Ndiffsize2}, together with \eqref{visizes} to obtain
\begin{equation}
\begin{split}
|I_3| &\lesssim  \| DN \|_{L^{2p}}^{2p} \| N \|_{L^\infty} \| DN \|_{L^\infty} \int_0 ^1 u^{-2\beta - 1 - \frac{1}{2p_0} } u^2 (1+u^{2\beta - 1/2} ) du,\\
|I_6| &\lesssim  \|DN \|_{L^{2p} }^{2p} \| N \|_{L^\infty} \| DN \|_{L^\infty} \int_0 ^1 u^{-2\beta - 1 - \frac{1}{2p_0} } u^{2\beta + 3/2} u (1+u^{2\beta - 1/2} ) du
\end{split}
\end{equation}
which is precisely the integrals we had to bound when estimating $I_1$ and $I_4$. Therefore,\begin{equation}
|I_3| + |I_6| \lesssim \| DN\|_{L^{2p} }^{2p} \| N \|_{L^\infty} \| DN \|_{L^\infty}
\end{equation}
if $p_0$ satisfies \eqref{p0cond1} and \eqref{p0cond2}. 

\medskip

\noindent \underline{Further splitting of the main terms, $I_2$ and $I_5$.} We can write
\begin{equation}
\begin{split}
I_2 &= I_{2,1} + I_{2,2}, \qquad I_5 = I_{5,1} + I_{5,2}, \\
I_{2,1} &= \int_0 ^\infty \int_0 ^1 (DN)^{2p-1} (\omega) W(u) \sum_{k=1} ^2 (-1)^k v_k^{2\beta - 1/2} v_2^{2\beta+3/2} (DN_{1,k} - DN_{2,k}) N_{3,k} N_{4,k} \, du \, d\omega, \\
I_{2,2} &= \int_0 ^\infty \int_0 ^1 (DN)^{2p-1} (\omega) W(u) \sum_{k=3} ^4 (-1)^k v_k^{2\beta - 1/2} v_2^{2\beta+3/2} (DN_{1,k} - DN_{2,k}) N_{3,k} N_{4,k} \, du \, d\omega, \\
I_{5,1} &= \int_0 ^\infty \int_0 ^1 (DN)^{2p-1} (\omega) W(u) \sum_{k=1} ^2 (-1)^k v_k^{2\beta - 1/2} v_4^{2\beta+3/2} (DN_{3,k} - DN_{4,k} ) N_{1,k} N_{2,k} \, du \, d\omega, \\
I_{5,2} &= \int_0 ^\infty \int_0 ^1 (DN)^{2p-1} (\omega) W(u) \sum_{k=3} ^4 (-1)^k v_k^{2\beta - 1/2} v_4^{2\beta+3/2} (DN_{3,k} - DN_{4,k} ) N_{1,k} N_{2,k} \, du \, d\omega.
\end{split}
\end{equation}

\medskip

\noindent \underline{Estimate of $I_{2,2}$ and $I_{5,1}$.} It is achieved through integration by parts. First, we see that if $u>0$ is sufficiently away from $0$, then an estimate is easily obtained. Indeed, for $\eta \in (0,1)$ to be chosen later, and for $u \in [\eta, 1]$,
\begin{equation}
\left |W(u) \sum_{k=3}^4 (-1)^k v_k ^{2\beta - 1/2} v_2^{2\beta+3/2} (DN_{1,k} - DN_{2,k} ) N_{3,k} N_{4,k} \right | \le C(\eta) \|N\|_{L^\infty}^2 (|DN_{1,k}| + |DN_{2,k} |),
\end{equation}
with a similar bound for $I_{5,1}$. Applying the previous argument (H\"older's inequality, Minkowski's integral inequality, and change of variables for $DN_{1,k}, DN_{2,k} $) gives that the portion of $I_{2,2}$ and $I_{5,1}$ corresponding to $u \in [\eta, 1]$ is bounded by
\begin{equation}
C(\eta, \beta, p_0) \| N \|_{L^\infty}^2 \| DN \|_{L^{2p} }^{2p}.
\end{equation}
Again, the coefficients may depend on $\eta, \beta, p_0$, but they are uniform in $p$. Next, we consider the domain $u \in [0, \eta]$. By the Newton-Leibniz formula,
\begin{equation}
\begin{split}
DN_{1,k} - DN_{2,k} &= \int_0 ^1 \frac{\omega (v_1 - v_2) }{v_k} (DN)' \left ( \frac{\omega}{v_k} (\lambda v_1 + (1-\lambda) v_2 \right ) d \lambda \\
&=\frac{v_1 - v_2}{v_k} \int_0^1 \frac{\frac{d}{du} \left [ DN\left ( \frac{\omega}{v_k} (\lambda v_1 + (1-\lambda) v_2 \right ) \right ]  }{\frac{d}{du} \left [\frac{\lambda v_1 + (1-\lambda) v_2}{v_k } \right ]} d \lambda,
\end{split}
\end{equation}
so that, after integration by parts in $u$, the portion of $I_{2,2}$ corresponding to $u \in [0, \eta]$ can be written as
\begin{equation} \label{I22IbP}
\begin{split}
I_{2,2} &= \int_0 ^\infty d \omega  (DN)^{2p-1} (\omega) \sum_{k=3}^4 \int_0^1 d \lambda \frac{W(u)  (-1)^k v_k^{2\beta - 1/2} v_2^{2\beta+3/2} }{\frac{d}{du} \left [\frac{\lambda v_1 + (1-\lambda) v_2}{v_k } \right ]} \frac{v_1 - v_2} {v_k} N_{3,k} N_{4,k}  \\
& \qquad \left. \times \left [ DN\left ( \frac{\omega}{v_k} (\lambda v_1 + (1-\lambda) v_2 \right ) \right ] \right |_{u=0} ^{u=\eta} \\
&-\int_0 ^\infty d \omega  (DN)^{2p-1} (\omega) \sum_{k=3}^4 \int_0^1 d \lambda \int_0 ^\eta \frac{d}{du} \left [ \frac{W(u)  (-1)^k v_k^{2\beta - 1/2} v_2^{2\beta+3/2} }{\frac{d}{du} \left [\frac{\lambda v_1 + (1-\lambda) v_2}{v_k } \right ]} \frac{v_1 - v_2} {v_k} N_{3,k} N_{4,k} \right ] \\
& \qquad \times \left [ DN\left ( \frac{\omega}{v_k} (\lambda v_1 + (1-\lambda) v_2 \right ) \right ] du.
\end{split}
\end{equation}
We first investigate the term $\frac{d}{du} \left [ \frac{\lambda v_1 + (1-\lambda) v_2}{v_k} \right ]$. Since $\frac{\lambda v_1 + (1-\lambda) v_2}{v_k}$, $k=3,4$ can be written as
\begin{equation}
u^{-1} (\lambda A_1 (u) + (1-\lambda) A_2 (u) ) =: u^{-1} A_\lambda (u),
\end{equation}
where $A_1, A_2$ are analytic functions in $[0,1]$, satisfying 
\begin{equation} \label{Aders}
A_i (u) - u A_i ' (u), A_i (u) - u A_i '(u) + \frac{1}{2} u^2 A_i ''(u)  \in \left (\frac{1}{2}, 2 \right ) , u \in [0, \eta], i= 1, 2,
\end{equation}
(here we choose $\eta>0$ sufficiently small), therefore the same bounds hold for $A_\lambda$ for all $\lambda \in [0,1]$ and $u \in [0, \eta]$. It is clear that $\eta>0$ is universal and independent of $\lambda$ and $p$. Now we see that $\frac{v_k} {\lambda v_1 + (1-\lambda) v_2 } \lesssim u$ uniformly in $\lambda$, and

\begin{equation}
\begin{split}
\frac{d}{du} \left [ \frac{\lambda v_1 + (1-\lambda) v_2}{v_k} \right ] &= -u^{-2} ( A_\lambda(u) - u A_\lambda ' (u) ), \\
\frac{d^2}{du^2} \left [ \frac{\lambda v_1 + (1-\lambda) v_2}{v_k} \right ] &= 2u^{-3} \left (A_\lambda(u) - u A_\lambda' (u) + \frac{1}{2} u^2 A_\lambda ''(u) \right),
\end{split}
\end{equation}
and therefore 
\begin{equation}
\left | \frac{d}{du} \left [ \frac{\lambda v_1 + (1-\lambda) v_2}{v_k} \right ] \right | \sim u^{-2}, \quad \left | \frac{d^2}{du^2} \left [ \frac{\lambda v_1 + (1-\lambda) v_2}{v_k} \right ] \right | \sim u^{-3}.
\end{equation}
Next, we investigate the size of the other terms: for $k=3, 4$
\begin{equation}
  W(u) (-1)^{k} v_{k} ^{2\beta - 3/2} v_2^{2\beta+3/2} (v_1 - v_2) = u^{-2\beta-1} 1^{2\beta+3/2} u^2 u^{2\beta -3/2} C_{k}(u) = u^{-1/2} C_{k}(u) ,
\end{equation}
where $C_{k}(u)$ is an analytic function which are bounded away from 0 in $u \in [0, \eta]$ (we can adjust $\eta>0$ to satisfy this). Therefore, we have, for $k=3, 4$,
\begin{equation}
\begin{split}
&\left |  W(u) (-1)^{k} v_{k} ^{2\beta - 3/2} v_2^{2\beta+3/2} (v_1 - v_2) \right | \lesssim  u^{-1/2}, \\
&\left | \frac{d}{du}  W(u) (-1)^{k} v_{k} ^{2\beta - 3/2} v_2^{2\beta+3/2} (v_1 - v_2) \right | \lesssim u^{-3/2}.
\end{split}
\end{equation}

Again, since $C_{k}$, $k=3,4$, is independent of $p$, the estimates are uniform in $p$. Now we estimate the first term of \eqref{I22IbP}. Assuming that $N, DN \in L^\infty$, since 
$$
\left |\sum_{k=3} ^4 \frac{W(u)  (-1)^k v_k^{2\beta - 1/2} v_2^{2\beta+3/2} }{\frac{d}{du} \left [\frac{\lambda v_1 + (1-\lambda) v_2}{v_k } \right ]} \frac{v_1 - v_2} {v_k} \right |\lesssim u^{3/2},
$$
the contribution for $u=0$ vanishes. The contribution for $u = \eta$ can be estimated in the usual way: H\"older inequality, Minkowski integral inequality (this time we switch $d\lambda$ and $d\omega$), and change of variables. Then we have the bound
\begin{equation}
C(\eta, \beta) \|DN \|_{L^{2p}} ^{2p} \| N\|_{L^\infty}^2
\end{equation}
for the first term of \eqref{I22IbP}. For the second term of \eqref{I22IbP}, we see that, if $k=3,4$,
\begin{equation}
\begin{split}
&\left | \left [ \frac{W(u)  (-1)^k v_k^{2\beta - 1/2} v_2^{2\beta+3/2} }{\frac{d}{du} \left [\frac{\lambda v_1 + (1-\lambda) v_2}{v_k } \right ]} \frac{v_1 - v_2} {v_k}  \right ] \right | \lesssim u^{3/2}, \\
&\left | \frac{d}{du} \left [ \frac{W(u)  (-1)^k v_k^{2\beta - 1/2} v_2^{2\beta+3/2} }{\frac{d}{du} \left [\frac{\lambda v_1 + (1-\lambda) v_2}{v_k } \right ]} \frac{v_1 - v_2} {v_k}  \right ] \right | \lesssim  u^{1/2}, \\
\end{split}
\end{equation}
and
\begin{equation}
\frac{d}{du} N_{i,k} = DN_{i,k} \frac{\frac{d}{du} (v_i/v_k) }{v_i/v_k}.
\end{equation}
If $i,k \in \{3,4\}$, we immediately see that
\begin{equation} \label{vivk34}
\left | \frac{\frac{d}{du} (v_i/v_k) }{v_i/v_k} \right |= \left |\frac{d}{du} \log (v_i/v_k) \right | \lesssim 1 ,
\end{equation}

Putting these estimates together, and using H\"older inequality, Minkowski integral inequality, and change of variables, the second term of \eqref{I22IbP} can be bounded by
\begin{equation}
C(\beta) \| DN \|_{L^{2p}}^{2p} \sum_{k=3,4} \int_0^\eta du \int_0 ^1 d\lambda \| N \|_{L^\infty} (\| N \|_{L^\infty} + \| DN \|_{L^\infty} ) u^{1/2}  \left ( \frac{v_k}{\lambda v_1 + (1-\lambda) v_2 } \right )^{\frac{1}{2p} }.
\end{equation}

 Applying the same argument as before, if $p_0$ satisfies \eqref{p0cond1}, it is bounded by
\begin{equation} \label{I22IbPbound}
C(\beta, \eta, p_0) \| DN\|_{L^{2p} }^{2p} \| N \|_{L^\infty} ( \| N \|_{L^\infty} + \| DN \|_{L^\infty})
\end{equation}
and again the coefficient is uniform in $p$. 

The term $I_{5,1}$ is estimated in a similar manner: again we split the domain into the regions $u \in [0, \eta]$ and $u \in [\eta, 1]$, and  the region $u\in [\eta, 1]$ is treated in exactly the same manner. Instead of \eqref{vivk34}, we have the following: if $i, k \in \{1,2\}$,
\begin{equation} \label{vivk12}
\left | \frac{\frac{d}{du} (v_i/v_k) }{v_i/v_k} \right | \lesssim u.
\end{equation}

For the region $u \in[0, \eta]$, we note that $\frac{\lambda v_3 + (1-\lambda) v_4}{v_k}$, $k=1,2$ can be written as 
\begin{equation}
u (\lambda B_1 (u) + (1-\lambda) B_2 (u) ) =: u B_\lambda (u),
\end{equation}
and the counterpart of \eqref{Aders} for $B_i$s hold for sufficiently small $\eta$ as well. Thus, we have
\begin{equation}
\left| \frac{\lambda v_3 + (1-\lambda) v_4}{v_k} \right | \sim u, \quad\left | \frac{d}{du} \left [ \frac{\lambda v_3 + (1-\lambda) v_4}{v_k} \right ] \right | \sim 1, \quad \left | \frac{d^2}{du^2} \left [ \frac{\lambda v_3 + (1-\lambda) v_4}{v_k} \right ] \right | \lesssim 1
\end{equation}
uniformly in $\lambda$.
Also, for $k=1, 2$
\begin{equation}
 W(u) (-1)^k v_k^{2\beta - 3/2} v_4^{2\beta+3/2} (v_3 - v_4) = u^{-2\beta - 1} u^{2\beta+3/2} u^2 (1^{2\beta - 3/2} C_{k}(u) ) = u^{5/2} C_{k} (u) 
\end{equation}
for an analytic function $C_{k}(u)$ which is bounded away from $0$ in $u \in [0, \eta]$, and we have, if $k=1,2$, 
\begin{equation}
\begin{split}
& \left | \left [ \frac{W(u) (-1)^k v_k^{2\beta - 1/2} v_4^{2\beta+3/2} }{\frac{d}{du} \left [ \frac{\lambda v_3 + (1-\lambda) v_4 }{v_k } \right ] } \frac{v_3 - v_4}{v_k} \right ] \right | \lesssim u^{5/2} , \\
& \left | \frac{d}{du} \left [ \frac{W(u) (-1)^k v_k^{2\beta - 1/2} v_4^{2\beta+3/2} }{\frac{d}{du} \left [ \frac{\lambda v_3 + (1-\lambda) v_4 }{v_k } \right ] } \frac{v_3 - v_4}{v_k} \right ] \right | \lesssim u^{3/2} .
\end{split}
\end{equation}
Then, the same arguments as for $I_{2,2}$ gives the same bound as in \eqref{I22IbPbound}. 

\medskip

\noindent \underline{Estimate of $I_{2,1}$ and $I_{5,2}$.} It turns out that these terms involve higher derivatives in an essential way, but they have a good sign, which can be interpreted as parabolicity. We show the estimates for $I_{5,2}$, the case $I_{2,1}$ being essentially identical. We write
\begin{equation}
\begin{split}
I_{5,2} &= \int_0 ^\infty \int_0 ^1 (DN)^{2p-1} (\omega) W(u) v_4^{2\beta-1/2} v_4^{2\beta+3/2} (DN_{3,4} - DN_{4,4} ) N_{1,4} N_{2,4} \,du \,d\omega \\
&- \int_0 ^\infty \int_0 ^1 (DN)^{2p-1} (\omega) W(u) v_3^{2\beta - 1/2} v_4^{2\beta+3/2} (DN_{3,3} - DN_{4,3}) N_{1,3} N_{2,3} \,du \, d\omega.
\end{split}
\end{equation}
We apply the change of variables
\begin{equation}
\overline{\omega} = \omega \frac{v_4}{v_3},
\end{equation}
to the second integral, which leads to 
\begin{equation}
\begin{split}
& I_{5,2} = \int_0 ^\infty \int_0 ^1 (DN)^{2p-1} (\omega) W(u) v_4^{2\beta-1/2} v_4^{2\beta+3/2} (DN_{3,4} - DN_{4,4} ) N_{1,4} N_{2,4} \,du \,d\omega \\
& \qquad \qquad - \int_0 ^\infty \int_0 ^1 (DN_{3,4} )^{2p-1}  W(u) v_3^{2\beta - 1/2} v_4^{2\beta+3/2} (DN_{3,4} - DN_{4,4}) N_{1,4} N_{2,4} \,du \frac{v_3}{v_4}\,d\omega \\
&=\int_0 ^\infty \int_0 ^1 \left [ (DN_{4,4})^{2p-1}  - (DN_{3,4} )^{2p-1} \right ] W(u) v_4^{2\beta-1/2} v_4^{2\beta+3/2} (DN_{3,4} - DN_{4,4} ) N_{1,4} N_{2,4} \,du\, d\omega \\
& - \int_0 ^\infty \int_0^1 (DN_{3,4})^{2p-1} W(u) (v_3^{2\beta+1/2}v_4^{2\beta+1/2} - v_4^{2\beta+1/2} v_4^{2\beta+1/2} ) (DN_{3,4} - DN_{4,4} ) N_{1,4} N_{2,4} \,du \,d\omega \\
&= - \int_0 ^\infty \int_0^1 \left [ (DN_{4,4})^{2p-1}  - (DN_{3,4} )^{2p-1} \right ] \left [ DN_{4,4} - DN_{3,4} \right ] W(u) v_4^{4\beta +1} N_{1,4} N_{2,4} \,du \,d\omega \\
&- \int_0^\infty \int_0^1 (DN(\omega))^{2p-1} (DN_{3,3} - DN_{4,3} ) N_{1,3} N_{2,3} W(u) v_4^{2\beta+1/2} ( v_3^{2\beta+1/2} - v_4^{2\beta+1/2} ) \,du\, d\omega,
\end{split}
\end{equation}
where we changed the integration variable once more in the last line above.

The first term is negative, due to nonnegativity of $N$ and monotonicity of $x \rightarrow x^{2p-1}$. The second term is controlled in the usual way: Holder, Minkowski integral inequality, and change of variables. The computation is essentially same as $I_4$. To summarize, we have the following:

\begin{lemma} \label{Lipschitz_apriori}
For every positive integer $p>p_0$, we have
\begin{equation} \label{parabolic_estimate}
\begin{split}
&\frac{d}{dt} \| DN \|_{L^{2p}} + \|DN\|_{L^{2p} }^{-(2p-1)}\int_0 ^\infty \int_0^1 \left [ (DN_{4,4})^{2p-1}  - (DN_{3,4} )^{2p-1} \right ] \\ & \qquad \qquad \qquad \qquad \qquad \left [ DN_{4,4} - DN_{3,4} \right ] W(u) v_4^{2\beta - 1/2} v_4^{2\beta+3/2} N_{1,4} N_{2,4} \, du \, d\omega \\
&+ \|DN\|_{L^{2p} }^{-(2p-1)}\int_0 ^\infty \int_0^1 \left [ (DN_{2,2})^{2p-1}  - (DN_{1,2} )^{2p-1} \right ] \\
& \qquad \qquad \qquad \qquad \qquad \left [ DN_{2,2} - DN_{1,2} \right ] W(u) v_2^{2\beta - 1/2} v_2^{2\beta+3/2} N_{3,2} N_{4,2} \, du \, d\omega\\
&\le \mathbf{C}(\beta, p_0) \| DN \|_{L^{2p} }  \| N \|_{L^\infty} (\| N \|_{L^\infty} + \| DN \|_{L^\infty} ).
\end{split}
\end{equation}
Here, $\mathbf{C}(\beta, p_0)$ depends only on $\beta$ and $p_0$ and in particular uniform in $p$. 
\end{lemma}

Combining with \eqref{LinftyN}, for every $p > p_0$ we have
\begin{equation}
\frac{d}{dt} \left ( \| N \|_{L^\infty} + \| DN \|_{L^{2p} } \right ) \le \mathbf{C} (\beta, p_0) (\| N \|_{L^\infty} + \| DN \|_{L^\infty } ) \| N \|_{L^\infty} \left ( \| N \|_{L^\infty} + \| DN \|_{L^{2p} } \right ),
\end{equation}
and thus
\begin{equation}
\begin{split}
&\| N (t) \|_{L^\infty} + \| DN (t) \|_{L^{2p } }  \\
& \qquad \le \left ( \| N (0) \|_{L^\infty} + \| DN (0) \|_{L^{2p } } \right ) \exp \left ( \int_0 ^t \mathbf{C} (\beta, p_0) (\| N(s) \|_{L^\infty} + \| DN(s) \|_{L^\infty } ) \| N (s) \|_{L^\infty} ds\right )
\end{split}
\end{equation}
and by taking the limit $p \rightarrow \infty$, 
\begin{equation}
\begin{split}
\| N (t) \|_{L^\infty} + \| DN (t) \|_{L^\infty} &\le \left ( \| N_0 \|_{L^\infty}  + \| DN_0 \|_{L^\infty} \right ) \\
&\times \exp \left ( \int_0 ^t \mathbf{C} (\beta, p_0) (\| N(s) \|_{L^\infty} + \| DN(s) \|_{L^\infty } ) \| N (s) \|_{L^\infty} ds\right )
\end{split}
\end{equation}
and therefore we can close the a priori estimate. Also, if we define $F(t)$ by
\begin{equation}
F(t) = \left ( \| N_0 \|_{L^\infty}  + \| DN_0 \|_{L^\infty} \right ) \exp  \left ( \int_0 ^t \mathbf{C} (\beta, p_0) F(s)^2 ds  \right )
\end{equation}
then $F(t) $ is a supersolution to $\| N(t) \|_{L^\infty} + \| DN(t) \|_{L^\infty}$, and satisfies 
\begin{equation}
F' (t) = \mathbf{C} (\beta, p_0) F(t) ^3
\end{equation}
whose solution exists for a positive time $T = T(\beta, p_0, N_0 )$:
\begin{equation}
\sup_{t \in (0, T) } \| DN (t) \|_{L^\infty}  \le \mathbf{C} (\beta, p_0, N_0 ),
\end{equation}
and for every positive integer $p\ge p_0$
\begin{equation} \label{smoothingterm}
\begin{split}
&\int_0 ^{T} \int_0 ^\infty \int_0 ^1 \left [ (DN _{4,4})^{2p-1}  - (DN_{3,4} )^{2p-1} \right ] \left [ DN_{4,4} - DN_{3,4} \right ] W(u) v_4^{2\beta - 1/2} v_4^{2\beta+3/2} N_{1,4} N_{2,4} du d\omega dt \\
& + \int_0 ^{T} \int_0 ^\infty \int_0^1 \left [ (DN_{2,2})^{2p-1}  - (DN_{1,2} )^{2p-1} \right ] \left [ DN_{2,2} - DN_{1,2} \right ] W(u) v_2^{2\beta - 1/2} v_2^{2\beta+3/2} N_{3,2} N_{4,2} du d\omega dt \\
&\le C(N_0, p, \beta, p_0) < \infty. 
\end{split}
\end{equation}
The positivity and the following change of variable simplifies \eqref{smoothingterm}: we set
$$v = \frac{u^2}{1+u+u^2}.$$
Then $v$ is an increasing function for $u \in [0, 1]$, $v \in [0,1/3]$, $DN_{1,2} = DN (\omega (1-v))$, $du = \frac{(1+u+u^2)^2}{u(2+u)} dv$, and $u^2 \sim v$. Thus, we have
\begin{equation} \label{Gainsimplified1}
\begin{split}
\int_0 ^{T} \int_0 ^\infty \int_0 ^1 & | DN(\omega) - DN (\omega(1+u) ) |^2  u^{2\beta}  \\
\times &\left [\int_0 ^1 | \lambda DN (\omega) + (1-\lambda) DN (\omega (1+u) ) |^{2(p-1) } d\lambda \right]  du d\omega dt \\
 + \int_0 ^{T} \int_0 ^\infty \int_0 ^{1/3} & | DN(\omega) - DN (\omega(1-u) ) |^2   u^{-(\beta+1) }  \\
\times &\left [\int_0 ^1 | \lambda DN (\omega) + (1-\lambda) DN (\omega (1+u) ) |^{2(p-1) } d\lambda \right]  du d\omega dt \\
&\le C(N_0, p, \beta, p_0) < \infty, \text{ for every positive integer } p \ge p_0.
\end{split} 
\end{equation}
In particular, for $\beta \in (-3/4, 3/4)$, by choosing $p=1$, we obtain
 \begin{equation} \label{Gainsimplified2}
 \begin{split}
& \int_0 ^{T} \int_0 ^\infty \int_{-1/3} ^{1 }  | DN(\omega) - DN(\omega(1+u) )|^2 \left ( \mathbf{1}_{[-1/3, 0]} (u) |u|^{-(\beta+1)} + \mathbf{1}_{[0, 1]} (u) |u|^{2\beta} \right )  du d\omega dt \\
& \qquad \le C(N_0, \beta, p_0)< \infty.
\end{split}
\end{equation}
Finally, letting $\omega' = \omega (1+u)$, and symmetrizing the norm of \eqref{Gainsimplified2} in $\omega$ and $\omega'$, it is equivalent to \eqref{fractionalSobolevnorm} and we obtain \eqref{gainofregularityestimate}.

To summarize, we have the following a priori estimate:
\begin{lemma}
There exists a time $T = T(N_0, p_0, \beta)$ and constant $\mathbf{C} (\beta, p_0, N_0)$ such that 
\begin{equation}
\| N \|_{L^\infty (0, T; X) } \le \mathbf{C}(\beta, p_0, N_0 ) .
\end{equation}
Moreover, we have \eqref{Gainsimplified1} and for $\beta \in (-3/4, 3/4)$, we have \eqref{gainofregularityestimate}.
\end{lemma}

\subsection{Local well-posedness}
Let $p_0 > \max \left ( \frac{1}{4(1-\beta)}, \frac{1}{4(1+\beta) } \right )$ and recall that
\begin{equation} \label{Xspace}
X = \{ N \in L^\infty ((0, \infty )) | \| DN \|_{L^{2p_0} ((0, \infty  )) } + \| DN \|_{L^\infty((0, \infty ))} < \infty \}
\end{equation}
with the natural norm
\begin{equation}
\| N \|_{X} = \| N \|_{L^\infty ((0, \infty ))} + \| DN \|_{L^{2p_0} ((0, \infty  )) } + \| DN \|_{L^\infty((0, \infty ))}.
\end{equation}

\medskip

\noindent \underline{The approximation scheme.} For a $N_0 \in X, N_0 \ge 0$, we consider the following approximate problem:
\begin{equation} \label{appeqn}
\begin{split}
\partial_t N^\varepsilon (t, \omega) &= \overline{C}_\varepsilon (N^\varepsilon) (\omega) :=  \int_\varepsilon ^1  \frac{\left (\prod_{i=1} ^4 v_i \right )^{-\beta - \frac{1}{2} } }{1+u+u^2} \sum_{j,k=1} ^4  (-1)^{j+k} v_j ^{2\beta+ \frac{3}{2} } v_k^{2\beta - \frac{1}{2} } \frac{\prod_{i=1} ^4 N_{i,k} ^\varepsilon }{N_{j,k} ^\varepsilon } du, \\
N^\varepsilon (0, \omega) &= N_0 (\omega).
\end{split}
\end{equation}
The problem \eqref{appeqn} admits a unique solution $N^\varepsilon \in L^\infty (0, T_{\varepsilon} ; X)$ for some time $T_\varepsilon > 0$: this can be regarded as an ODE with values in $X$: we obviously have
\begin{equation}
\begin{split}
 \| \overline{C}_{\varepsilon} (N^{\varepsilon} ) \|_{L^\infty ((0, \infty )) } &\le C_{\varepsilon} \| N^{\varepsilon} \|_{L^\infty ((0, \infty )) } ^3 , \\
\| D \overline{C}_{\varepsilon} (N^{\varepsilon} ) \|_{L^\infty ((0, \infty )) } &\le C_{\varepsilon}  \| D N^{\varepsilon} \|_{L^\infty ((0, \infty )) }  \| N^{\varepsilon} \|_{L^\infty ((0, \infty )) } ^2,
 \end{split}
\end{equation}
and
\begin{equation}
\begin{split}
\| D \overline{C}_{\varepsilon} (N^{\varepsilon} ) \|_{L^{2p_0} ((0, \infty )) } &\le C_{\varepsilon}  \sum_{j,k = 1} ^4 \left  \| \int_{\varepsilon} ^1  D \left ( \prod_{i=1, i\ne j} ^4 N_{i,k} ^{\varepsilon} \right ) du \right \|_{L^{2p_0}((0, \infty )) } \\
&\le C_{\varepsilon, p_0} \| DN^{\varepsilon} \|_{L^{2p_0} ((0, \infty ) ) } \| N^{\varepsilon} \|_{L^\infty ((0, \infty )) } ^2.
\end{split}
\end{equation}
To sum up, we have $\| \partial_t N^{\varepsilon } \|_{X} \le C_{\varepsilon, p_0} \| N^{\varepsilon} \|_{X}^3$, so we have local well-posedness for $N^{\varepsilon}$ in $X$.

\medskip

\noindent \underline{Uniform bounds in $\epsilon$.}
The key observation is that all the estimates in Subsection \ref{Nonnegativity}, Subsection \ref{Linfty}, and Subsection \ref{Lipschitz} carry over in the exact same manner: we just replace $u \in [0,1]$ by $u \in [\varepsilon, 1]$. The only nontrivial part is the treatment of  $I_{2,2}$, $I_{5,1}$, where we have employed integration by parts, so the contribution of $u = \varepsilon$ might not be zero. However, the contribution for $u=\varepsilon$ is smaller than the contribution for $u = \eta$: in the case of $I_{2,2}$, the corresponding term is of the form $O(u^{3/2}) N^2 DN$. In particular, this implies that all $N^{\varepsilon}$ can be extended to the time $T = T(N_0)>0$ which is uniform in $\varepsilon$, and moreover we have uniform estimates
\begin{equation}
\sup_{\varepsilon \in (0, 1) } \| N^{\varepsilon} \|_{L^\infty (0, T ; X) } \lesssim 1.
\end{equation}
In addition to that, we also have that for every $p > p_0$, 
\begin{equation} \label{parabolicapprox}
\begin{split}
&\sup_{\varepsilon \in (0, 1) }\int_0 ^{T} \int_0 ^\infty \int_\varepsilon^1 \left [ (DN ^{\varepsilon}_{4,4})^{2p-1}  - (DN^{\varepsilon}_{3,4} )^{2p-1} \right ] \\
& \qquad \qquad \qquad \qquad \qquad \left [ DN^{\varepsilon}_{4,4} - DN^{\varepsilon}_{3,4} \right ] W(u) v_4^{2\beta - 1/2} v_4^{2\beta+3/2} N^{\varepsilon}_{1,4} N^{\varepsilon}_{2,4} \, du \, d\omega \, dt \\
& \qquad + \int_0 ^{T} \int_0 ^\infty \int_\varepsilon^1 \left [ (DN^{\varepsilon}_{2,2})^{2p-1}  - (DN^{\varepsilon}_{1,2} )^{2p-1} \right ]  \\
& \qquad \qquad \qquad \qquad \qquad \left [ DN^{\varepsilon}_{2,2} - DN^{\varepsilon}_{1,2} \right ] W(u) v_2^{2\beta - 1/2} v_2^{2\beta+3/2} N^{\varepsilon}_{3,2} N^{\varepsilon}_{4,2} \, du \, d\omega  \,dt \\
&\le C_{p, p_0, \beta}.
\end{split}
\end{equation}

\medskip

\noindent \underline{Convergence of a subsequence in the uniform topology.}
To apply a compactness argument: we first rely on the classical Arzela-Ascoli theorem. Fix $M >0$. We see that the collection of functions
\begin{equation}
 \{ N^{\varepsilon} (t, e^s),\, 0 \le t \le T,\, -M \le s \le M \}
\end{equation}
are uniformly bounded (by $\| N^{\varepsilon} \|_{L^\infty (0, T ; L^\infty ((0, \infty) ) }$) and uniformly equicontinuous: we have
\begin{equation}
\begin{split}
|N^{\varepsilon} (t, e^{s_2}) - N^{\varepsilon} (t, e^{s_1} )| &= \left | \int_{e^{s_1}} ^{e^{s_2} } DN^{\varepsilon} (\omega) \frac{d\omega}{\omega} \right | = \left | \int_{s_1} ^{s_2} DN^{\varepsilon} (e^s) ds \right | \\
& \le |s_2 - s_1| \| DN^{\varepsilon} \|_{L^\infty (0, T ; L^\infty ((0, \infty) ) }, \\ 
| N^{\varepsilon} (t_2, e^{s} ) - N^{\varepsilon} (t_1, e^s )  | &= \left | \int_{t_1} ^{t_2} \overline{C}_{\varepsilon} (N^{\varepsilon} (s) ) ds \right | \\
&\lesssim \| N^{\varepsilon } \|_{L^\infty (0, T ; L^\infty ) }^2  ( \| N^{\varepsilon } \|_{L^\infty (0, T ; L^\infty ) } + \| DN^{\varepsilon } \|_{L^\infty (0, T ; L^\infty ) } ) (t_2 - t_1)
\end{split}
\end{equation}
and thanks to the uniform estimates we obtain the desired conclusion. Therefore, by the Arzela-Ascoli theorem, we obtain a subsequence $\{ N^{\varepsilon_{k,M} }  \}_k$ which converges to a function $N$ uniformly in $[0, T] \times [e^{-M}, e^{M}]$, and by a diagonalization argument, we can extract a subsequence $N^{\varepsilon_k} $ which converges to $N$ uniformly on compact subsets of $[0, T] \times (0, \infty)$. In the remainder, we slightly abuse the notation and still denote this subsequence by $N^{\varepsilon}$. This also shows that $N \in L^\infty ((0, T; L^\infty (0, \infty ) )$. Also, this implies that $DN^{\varepsilon}$ converges to $DN$ in the sense of distribution in $[0, T] \times (0, \infty)$, and Banach-Alaoglu theorem applied to $L^\infty (0, T; L^\infty (0, \infty )) = L^1(0, T; L^1 (0, \infty) )^*$ and $L^\infty (0, T; L^{2p_0} (0, \infty )) = L^1 (0, T; L^{2p_0 ^*}(0, \infty ) )^*$ implies that up to subsequence $DN^{\varepsilon}$ is weak-$*$ convergent, and the limit should be therefore $DN$. Therefore, $N \in L^\infty (0, T; X)$. This also implies that $\overline{\mathcal{C} } (N) \in L^\infty(0, T; L^\infty ((0, \infty ) ) )$ and moreover, the integrand of $\overline{\mathcal{C} } (N)$ is in $L^1 ([0, T] \times [0, 1];  dt du)$ for almost every $\omega$, following the argument in \eqref{collisionpointwise}.

{\color{black} In the remainder of the argument, we shall abuse notations by denoting $N^{\epsilon}$ instead of $N^{\epsilon_k}$ the converging subsequence.} 

\medskip
\noindent
\underline{Solution in the mild sense.}
The next step is to show that $N$ is a mild solution, using the Vitali convergence theorem. Writing down \eqref{appeqn} in Duhamel form, we have
\begin{equation}
N^{\varepsilon} (t, \omega) - N^{\varepsilon} (0, \omega) = \int_0 ^t \overline{C}_{\varepsilon} (N^{\varepsilon} (s) ) (\omega) ds,
\end{equation}
and for any compact subset $ [0, T] \times C \subset [0, T] \times  (0, \infty) $ the left-hand side uniformly converges to $N(t, \omega) - N(0, \omega)$.
 For the right-hand side, we claim that for every $\omega \in (0, \infty)$, $\overline{C}_{\varepsilon} (N^{\varepsilon} (s) ) (\omega)$ converges to $\overline{C} (N(s) ) (\omega)$ in $L^1 ([0, t])$.

To show this, we first define 
\begin{equation}
\begin{split}
f^{\varepsilon} (u, s) &:= \mathbf{1}_{[\varepsilon, 1]} (u) \frac{\left (\prod_{i=1} ^4 v_i \right )^{-\beta - \frac{1}{2} } }{1+u+u^2} \sum_{j,k=1} ^4  (-1)^{j+k} v_j ^{2\beta+ \frac{3}{2} } v_k^{2\beta - \frac{1}{2} } \frac{\prod_{i=1} ^4 N_{i,k} ^\varepsilon (s) }{N_{j,k} ^\varepsilon (s) }, \\
f(u, s) &= \frac{\left (\prod_{i=1} ^4 v_i \right )^{-\beta - \frac{1}{2} } }{1+u+u^2} \sum_{j,k=1} ^4  (-1)^{j+k} v_j ^{2\beta+ \frac{3}{2} } v_k^{2\beta - \frac{1}{2} } \frac{\prod_{i=1} ^4 N_{i,k} (s)  }{N_{j,k} (s)  },
\end{split}
\end{equation}
with the measure space $([0,1] \times [0, t], du\, ds)$. We first show that $f^{\varepsilon}$ converges to $f$ in measure. Let $\zeta>0$ be an arbitrary small number. Then there exists a number $L(\zeta), U(\zeta)>0$ such that if $u \in [\zeta, 1]$, then $\omega v_{i,k} \in [L(\zeta), U(\zeta)]$ for all $\omega \in C$ and $i, k = 1,2,3, 4.$ Then for sufficiently small $\varepsilon>0$, due to uniform convergence of $N^{\varepsilon}$ on compact subsets, $f^{\varepsilon}$ converges uniformly to $f$ in $(u, s) \in [\zeta, 1] \times[0, t]$. Therefore, for sufficiently small $\varepsilon>0$, depending on $\zeta$, 
\begin{equation}
\{ (u,s) \in [0, 1] \times[0, t] | |f(u,s) - f^\varepsilon (u,s) | > \zeta \} \subset [0, \zeta] \times [0, t],
\end{equation}
which demonstrates that for each $\omega$ $f^\varepsilon$ converges to $f$ in measure in $(u,s) \in [0,1] \times [0, t]$. Next, we show that $\{ f^\varepsilon \}$ is uniformly integrable. By the same calculation as that of \eqref{collisionpointwise} we see that
\begin{equation}
\begin{split}
|f^{\varepsilon} (u,s) | &\le u^{-2\beta -1} (1+u^{2\beta - 1/2})(u^2 + u^{2\beta + 5/2} ) \| N^{\varepsilon } \|_{L^\infty (0, T ; L^\infty( (0, \infty) ))}^2 \\
& \qquad \qquad ( \| N^{\varepsilon } \|_{L^\infty (0, T ; L^\infty( (0, \infty) ))} + \| DN^{\varepsilon } \|_{L^\infty (0, T ; L^\infty( (0, \infty) ))} ),
\end{split}
\end{equation}
and by the uniform estimates we have
\begin{equation}
|f^{\varepsilon} (u,s) | \le C u^{-2\beta - 1} (1+u^{2\beta - 1/2} ) (u^2 + u^{2\beta+ 5/2}),
\end{equation}
and the right-hand side is integrable for $\beta \in (-1, 1)$. Thus, the uniform integrability condition is satisfied
\begin{equation}
\lim_{C \rightarrow \infty} \sup_{\varepsilon \in [0, 1]} \int_{ | f^{\varepsilon} > C | } |f^\varepsilon| du ds = 0.
\end{equation}
Therefore, the Vitali convergence theorem gives $L^1$ convergence, namely for every $\omega>0$,
\begin{equation}
\overline{C}_\varepsilon (N^\varepsilon (s) ) (\omega) \rightarrow \overline{C} (N (s) ) (\omega) \text{ in } L^1 ((0, t), ds).
\end{equation}
In conclusion, for each $\omega>0$, $N$ satisfies
\begin{equation} \label{mildsoln}
N(t, \omega) = N(0, \omega) + \int_0 ^t \overline{C} (N (s) )(\omega) ds
\end{equation}
so $N$ is a mild solution. 

\medskip

\noindent \underline{Solution in the strong sense.} We prove that the solution is in fact strong. We have already seen that for every $\omega$, the integrand of $\overline{\mathcal{C}} (N)$ is integrable in $[0, T] \times [0,1]$. In particular, by Fubini's theorem, for each $\omega \in (0, \infty)$, and for almost $t \in [0, T]$, $\overline{\mathcal{C}} (N(t) ) (\omega)$ is well-defined and is an integrable (and therefore measurable) function of $t \in [0, T]$. Therefore, by Lebesgue differentiation theorem, $\partial_t N (t, \omega) = \lim_{h \rightarrow 0} \frac{1}{h} \int_{t} ^{t+h} \overline{C} (N (s) )(\omega) ds$ exists and equals $\overline{C} (N(t) ) (\omega)$ for almost every $t \in [0, T]$. 

\medskip

\noindent \underline{Uniqueness of the solution.} Finally, we conclude that the solution to the Cauchy problem is unique, by establishing a stability estimate. Suppose that $\beta, p_0$ satisfies \eqref{betaanddata}, $N, M \in L^\infty (0, T; X)$ be two strong solutions of \eqref{Neqn} with initial data $N_0, M_0 \in X$ respectively, and $N_0 - M_0 \in L^{2p} ((0, \infty))$ for some positive integer $p \ge p_0$. If we let $\delta N = N - M$, we have
\begin{equation} \label{difference}
\begin{split}
\partial_t \delta N &= \int _0 ^1 W(u) \sum_{k=1} ^4 (-1)^k v_k ^{2\beta - \frac{1}{2} } \times \\
&\left [ (v_2 ^{2\beta + 3/2} - v_1^{2\beta + 3/2} ) \left ( (\delta N)_{2,k} N_{3,k} N_{4,k} + M_{2,k} (\delta N)_{3,k} N_{4,k} + M_{2,k} M_{3,k} (\delta N)_{4,k} \right ) \right. \\
&+ v_2^{2\beta + 3/2}  ( (\delta N)_{1,k}  -(\delta N)_{2,k} ) N_{3,k} N_{4,k} ) \\
&+ v_2^{2\beta + 3/2} (M_{1,k} - M_{2,k} ) \left ( (\delta N)_{3,k} N_{4,k} + M_{3,k} (\delta N)_{4,k} \right ) \\
&+ (v_4 ^{2\beta + 3/2} - v_3^{2\beta + 3/2} ) \left ( (\delta N)_{1,k} N_{2,k} N_{4,k} + M_{1,k} (\delta N)_{2,k} N_{4,k} + M_{1,k} M_{2,k} (\delta N)_{4,k} \right ) \\
&+ v_4^{2\beta + 3/2} ( (\delta N)_{3,k} - (\delta N)_{4,k} ) N_{1,k} N_{2,k} \\
&\left.+ v_4^{2\beta + 3/2} (M_{3,k} - M_{4,k} ) \left ( (\delta N)_{1,k} N_{2,k} + M_{1,k} (\delta N)_{2,k} \right ) \right ] du \\
&=: J_1 + J_2 + J_3 + J_4 + J_5 + J_6.
\end{split}
\end{equation}
Then we multiply $(\delta N)^{2p-1}$ to \eqref{difference} and integrate over $\omega$ to measure $\frac{1}{2p} \frac{d}{dt} \int_0 ^\infty (\delta N)^{2p} (\omega) d\omega$. The calculation proceeds exactly same as in Section \ref{Lipschitz}: for $\ell = 1,3,4,6$, 
\begin{equation}
\int_0 ^\infty (\delta N)^{2p-1} J_\ell d\omega \le \| \delta N \|_{L^{2p} (0, \infty) }^{2p-1} \| J_{\ell} \|_{L^{2p} (0,\infty)}
\end{equation}
and from \eqref{visizes}, \eqref{weight} and \eqref{vdiffsize} and by Minkowski integral inequality, we have
\begin{equation}
\begin{split}
&\int_0 ^\infty (\delta N)^{2p-1} (J_1 + J_4) d\omega \\
&\lesssim \| \delta N\|_{L^{2p}} ^{2p} (\| N \|_{L^\infty} + \| M \|_{L^\infty} )^2 \times \\
&  \int_0^1 u^{-2\beta -1} \sum_{k=1} ^4 v_k^{2\beta - 1/2} \left [u^2  (v_{k,2}^{1/2p} + v_{k,3}^{1/2p} + v_{k,4}^{1/2p} ) + u^{2\beta+5/2} (v_{k,1}^{1/2p} + v_{k,2}^{1/2p} + v_{k,4}^{1/2p} ) \right ] du \\
&\lesssim \| \delta N\|_{L^{2p}} ^{2p} (\| N \|_{L^\infty} + \| M \|_{L^\infty} )^2 \int_0 ^1 (u^{-2\beta+1} + u^{3/2} )(1+u^{-1/2p} ) + (u^{1/2} + u^{2\beta+1} )(1+u^{1/2p} ) du \\
&\lesssim \| \delta N\|_{L^{2p}} ^{2p} (\| N \|_{L^\infty} + \| M \|_{L^\infty} )^2
\end{split}
\end{equation}
by the conditions on $p, p_0, \beta$. Similarly, we have
\begin{equation}
\begin{split}
&\int_0 ^\infty (\delta N)^{2p-1} (J_3 + J_6) d\omega \lesssim \| \delta N \|_{L^{2p }}^{2p} ( \| N \|_{L^\infty} + \| M \|_{L^\infty} ) \| DM \|_{L^\infty} \times \\
& \int_0 ^1 u^{1-2\beta - 1/2p} + u^{3/2} + u^{1/2} + u^{1+2\beta +1/2p } du \lesssim  \| \delta N \|_{L^{2p }}^{2p} ( \| N \|_{L^\infty} + \| M \|_{L^\infty} ) \| DM \|_{L^\infty}.
\end{split}
\end{equation}
The estimates on $I_2$ and $I_5$ in Section \ref{Lipschitz}, mutatis mutandis, give the estimate on $J_2$ and $J_5$: we just replace $DN$ by $\delta N$ and we omit the details. To summarize, we have
\begin{equation}
\frac{d}{dt} \| \delta N \|_{L^{2p}} \lesssim \| \delta N \|_{L^{2p} } ( \| N \|_{L^\infty (0, T; X) } + \| M \|_{L^\infty (0, T; X) } )^2,
\end{equation}
which gives the uniqueness.

\end{document}